\documentclass{article}

\usepackage{arxiv}

\usepackage[utf8]{inputenc} 
\usepackage[T1]{fontenc}    
\usepackage[hidelinks]{hyperref}       
\usepackage{url}            
\usepackage{booktabs}       
\usepackage{amsfonts}       
\usepackage{nicefrac}       
\usepackage{microtype}      
\usepackage{lipsum}

\usepackage{amsmath}
\usepackage{amsthm}
\usepackage{url}
\usepackage{graphicx,subcaption,colortbl}
\usepackage{caption}
\usepackage{subcaption}
\usepackage{todonotes}
\usepackage{comment}
\newtheorem{assumption}{Assumption}
\usepackage{amssymb}
\usepackage{dirtytalk}
\usepackage{xcolor}
\newcommand{\ie}{{\emph{i.e., }}}
\newcommand{\eg}{{\emph{e.g., }}}
\newcommand{\ubar}[1]{\text{\b{$#1$}}}
\newcommand{\norm}[1]{\left\lVert#1\right\rVert}
\usepackage{amssymb}
\usepackage{mathtools}
\usepackage{algorithm}
\usepackage{longtable}
\usepackage{multirow}
\usepackage{tabularx}
\newcolumntype{Y}{>{\centering\arraybackslash}X}
\usepackage{dirtytalk}
\usepackage{rotating} 
\usepackage{pdflscape}

\theoremstyle{plain}
\newtheorem{theorem}{Theorem}
\newtheorem{corollary}{Corollary}
\newtheorem{lemma}{Lemma}
\newtheorem{proposition}{Proposition}

\theoremstyle{remark}
\newtheorem{remark}{Remark}

\theoremstyle{definition}
\newtheorem{definition}{Definition}

\newcommand{\R}{\mathbf{R}}

\begin{document}


\title{$P$-split formulations: A class of intermediate formulations between big-M and convex hull for disjunctive constraints}

\author{Jan Kronqvist\thanks{To whom all correspondence should be addressed.} \\Department of Mathematics\\ KTH Royal Institute of Technology\\ Stockholm, Sweden\\
\texttt{jankr@kth.se}
\And Ruth Misener\\
Department of Computing\\ Imperial College London
\\London, UK\\
\texttt{r.misener@imperial.ac.uk}
\And Calvin Tsay\\
Department of Computing\\ Imperial College London
\\London, UK\\
\texttt{c.tsay@imperial.ac.uk}}

\maketitle
\begin{abstract}
We develop a class of mixed-integer formulations for disjunctive constraints intermediate to the big-M and convex hull formulations in terms of relaxation strength. The main idea is to capture the best of both the big-M and convex hull formulations: a computationally light formulation with a tight relaxation. The ``$P$-split'' formulations are based on a lifted transformation that splits convex additively separable constraints into $P$ partitions and forms the convex hull of the linearized and partitioned disjunction. The ``$P$-split'' formulations are derived for disjunctive constraints with convex constraints within each disjunct, and we generalize the results for the case with nonconvex constraints within the disjuncts. We analyze the continuous relaxation of the $P$-split formulations and show that, under certain assumptions, the formulations form a hierarchy starting from a big-M equivalent and converging to the convex hull. We computationally compare the  $P$-split formulations against big-M and convex hull formulations on 344 test instances. The test problems include K-means clustering, semi-supervised clustering, P\_ball problems, and optimization over trained ReLU neural networks. The computational results show promising potential of the $P$-split formulations.  For many of the test problems, $P$-split formulations are solved with a similar number of explored nodes as the convex hull formulation, while reducing the solution time by an order of magnitude and outperforming big-M both in time and number of explored nodes.\end{abstract}


\section{Introduction}

Many optimization problems have a clear disjunctive structure, where one decision leads to some constraints and another decision results in another set of constraints. In fact, any convex mixed-integer program (MIP), whether linear or convex nonlinear, can be viewed as a disjunctive program with the single disjunctive constraint that the solution must be within at least one convex set out of a collection of convex sets, \ie within the union of the convex sets
\cite{Balas2018}. Disjunctive constraints are in general nonconvex, as the union of convex sets is generally not a convex set. Handling these nonconvex disjunctive constraints is a fundamental question in MIP. Disjunctive programming  was pioneered by Balas \cite{balas1974disjunctive,balas1979disjunctive,balas1985disjunctive,balas1998disjunctive} and generalized by Grossmann and coworkers \cite{grossmann2003generalized,lee2000new,raman1994modelling,ruiz2012hierarchy,Trespalacios2014}. In general, disjunctive programming offers an alternative view of MIP that may be valuable for analyzing problems, creating strong extended formulations \cite{balas1998disjunctive}, and deriving cutting planes \cite{balas1993lift,balas2022v,chen2024sparse,kronqvist2020disjunctive,trespalacios2016cutting}. However, the best encoding of a disjunctive program, or disjunctive constraint(s), as an MIP is an open question.

The classic alternatives for expressing a disjunctive constraint as a MIP, the \textit{big-M} and \textit{convex hull} formulations, come with well-known trade-offs between relaxation strength and problem size. Convex hull formulations \cite{balas1998disjunctive,ben2001lectures,ceria1999convex,helton2009sufficient,jeroslow1984modelling,stubbs1999branch} provide a \textit{sharp formulation} \cite{vielma2015mixed} for a single disjunction, \ie the continuous relaxation provides the best possible lower bound. The convex hull of a disjunction is often represented by an \emph{extended formulation} \cite{Balas2018,bonami2015mathematical,conforti2008compact,grossmann2003generalized,gunluk2010perspective,hijazi2012mixed,vielma2015mixed} which introduces multiple copies of each variable in the disjunction. Extended formulations are also called ``perspective'' or ``multiple-choice'' formulations. 
On the other hand, big-M formulations only introduce one binary variable for each disjunct, resulting in smaller problems in terms of both number of variables and constraints. The big-M formulation also preserves the constraint expressions, whereas the extended convex hull builds upon the perspective function, which can result in numerically more challenging constraints. In general, big-M formulations provide weaker continuous relaxations than the convex hull and may require a solver to explore significantly more nodes in a branch-and-bound tree \cite{conforti2014integer,vielma2015mixed}. 
But the computationally simpler subproblems of the big-M formulation can offset the larger number of explored nodes. Anderson et al.\ \cite{anderson2020strong} mention a folklore observation that extended convex hull formulations tend to perform worse than expected. The observation is supported by past numerical results \cite{anderson2020strong,kronqvist2021between,trespalacios2015algorithmic,tsay2021partition}, as well as this paper. 

This paper introduces a new class of formulations for encoding disjunctive constraints in a MIP framework, which we term \textit{$P$-split formulations}. The $P$-split formulations lie between the big-M and convex hull in terms of relaxation strength and attempt to combine the best of both worlds: a tight, yet computationally light, formulation. The main idea behind the $P$-split formulations is to perform a lifting transformation, partitioning the constraints of each disjunct to both linearize the disjunction and reduce the number of variables in the disjunction. Forming the convex hull of the resulting linear disjunctions results in a smaller problem, while retaining some key features of the convex hull. We call the formulations $P$-split, since the constraints are `split' into $P$ parts. Efficient convex hull formulations have been actively researched \cite{balas1988convex,bernal2024convex,conforti2008compact,jeroslow1988simplification,sawaya2007computational,trespalacios2015algorithmic,vielma2019small,vielma2011modeling,vielma2010mixed}, and several techniques for deriving the convex hull of MIPs have been presented \cite{balas1985disjunctive,lasserre2001explicit,lovasz1991cones,ruiz2012hierarchy,sherali1990hierarchy}. 
The motivation behind the $P$-split formulations is different: the goal is not to obtain a formulation whose continuous relaxation forms the convex hull of the disjunction. Instead, the $P$-split formulations provide a straightforward framework for generating a class of formulations whose relaxations approximate the convex hull (for a general class of disjunctive constraints) using a smaller, computationally more tractable, problem formulation. A hierarchy of $P$-split formulations can be obtained by adjusting $P$, where larger $P$ values result in stronger relaxations at the cost of increased problem size. Our computational results show that intermediate $P$-split formulations can offer a significant advantage over both the big-M and convex hull formulations. 
The $P$-split formulations are particularly well-suited for problems with many variables in the constraints of each disjunct, but with only a few constraints per disjunct. Problems with such a structure include, \eg machine learning applications such as clustering \cite{maimon2005data,papageorgiou2018pseudo} and optimizing over trained neural networks \cite{haddad2022verification,huchette2023deep,papalexopoulos2021constrained,patel2022neur2sp}. 

Related work includes research on deriving formulations stronger than big-M without introducing auxiliary variables, such as improved big-M \cite{trespalacios2015improved}, hybrid formulations \cite{vielma2015mixed}, and formulations based on the Cayley embedding \cite{vielma2019small}. Some formulations also build upon specific problems structures such as, indicator constraints \cite{bonami2015mathematical,hijazi2012mixed} and constraints with multiple right-hand sides \cite{jeroslow1988simplification}. A detailed overview of alternative formulations is given in the review paper \cite{vielma2015mixed}.

This paper builds on our conference paper \cite{kronqvist2021between} and offers a significant extension both in the fundamental theory and computational analysis. Specifically, we generalize the results of \cite{kronqvist2021between} to problems with multiple constraints per disjunct and nonconvex constraints, we generalize and refine the mathematical proofs, we provide a comprehensive analysis of the relaxation strength, we identify circumstances where a $P$-split formulation cannot recover the convex hull of the disjuctive constraint, we include extended computational tests, and we overall provide a more complete discussion of theoretical and practical details. The main result presented in \cite{kronqvist2021between}, that $P$-split formulations form a hierarchy in terms of relaxation strength, was only shown for a two-term disjunction with a single constraint per disjunct. 
The formulations for this case are now a Pyomo.GDP 
\cite{chen2021pyomo}\footnote{\url{https://github.com/Pyomo/pyomo/pull/2221}; accessed 05/2025} plugin. 
This work extends $P$-split formulations to general disjunctions with multiple constraints and/or disjunctive terms. We discuss how the formulation can be strengthened with locally valid bounds, \ie additional linear inequality constraints, and give implementation guidelines for the resulting formulations. We expand the numerical study from 12 to 344 instances \cite{kronqvist2021between}.  

In this paper, Section~\ref{sec:2} presents the background and assumptions upon which the $P$-split formulations are built. Section~\ref{sec:3} introduces the $P$-split formulations for the case with a single constraint per disjunct, which simplifies the derivations and analysis. Section~\ref{sec:4} covers the general case, together with properties of the $P$-split relaxations and how they compare to the big-M and convex hull relaxations. Section~\ref{sec:4} also presents linking constraints that can be used to further strengthen the formulations for disjuncts with multiple constraints, along with conditions for avoiding redundant linking constraints. Section~\ref{sec:5} shows that the $P$-split framework can directly be applied to disjunctive constraints with separable nonconvex constraint functions if used in conjunction with convex underestimators, and generalizes the main results from Section~\ref{sec:4}. 
Section~\ref{sec:6} describes some computational considerations. Section~\ref{sec:7} presents a numerical comparison of the formulations, using instances with both linear and nonlinear disjunctions.

\section{Background}\label{sec:2}
We consider mixed-integer (nonlinear), or MI(N)LP, problems with disjunctions
\begin{equation}
\begin{aligned}\label{eq:main_disjunction}
    &\underset{l \in \mathcal{D}}{\lor} \begin{bmatrix} g_{k}(\boldsymbol{x}) \leq b_{k} \quad \forall k \in \mathcal{C}_{l}\\
    \boldsymbol{x} \in \mathcal{X} 
    \end{bmatrix}\\
\end{aligned}
\end{equation}
where $\mathcal{X} \subset \mathbb{R}^n$ is a convex compact set, $\mathcal{D}$ contains the disjunct indices, and $\mathcal{C}_{l}$ the constraint indices in disjunct $l$. This paper focuses on modeling Disjunction \eqref{eq:main_disjunction}, which is embedded in a larger MINLP. So, forming the convex hull of \eqref{eq:main_disjunction} may not recover the convex hull of the entire mixed-integer problem.
To derive the $P$-split formulations, we make the following assumptions. 

\begin{assumption}
The functions $g_{k}: \mathbb{R}^n \rightarrow \mathbb{R} $ are convex additively separable functions, \ie  $g_{k}(\boldsymbol{x}) = \sum_{i=1}^n h_{i,k}(x_i)$ where $h_{i,k}: \mathbb{R} \rightarrow \mathbb{R} $ are convex functions, and each disjunct is non-empty.   
\end{assumption}
\begin{assumption}
All functions $g_k$ are bounded over $\mathcal{X}$.
\end{assumption}
\begin{assumption}
$P$-split formulations are implemented with a minimal number of auxiliary variables; see Remark 1 for more details.
\end{assumption}
\begin{assumption}
Each disjunct contains (significantly) fewer constraints than the number of variables in the disjunction, \ie $\left| \mathcal{C}_{l}\right| << n, \ \ \forall l \in \mathcal{D
} $.
\end{assumption}

Assumptions 1--4 ensure certain properties of the resulting $P$-split formulations and simplify the derivation. 
Applications satisfying Assumptions 1--4 include clustering \cite{papageorgiou2018pseudo,sauglam2006mixed}, mixed-integer classification \cite{liittschwager1978integer,rubin1997solving}, optimization over trained neural networks \cite{anderson2020strong,botoeva2020efficient,ceccon2022omlt,grimstad2019relu,serra2020lossless,papalexopoulos2021constrained,tsay2021partition}, and coverage optimization \cite{huang2005coverage}. 

Assumptions 1--4 are sufficient to apply a $P$-split formulation but not necessary. For example, Assumption 4 is technically not needed, but rather characterizes favorable problem structures. Assumption~1 simplifies our derivations, but it can be relaxed because the formulations do not require the constraint functions to be fully separable.  
Constraints that are \textit{partially} additively separable (see Definition 1 below) would suffice. 
Likewise, Section~5 shows that $P$-split formulations can be applied to \textit{non-convex} additively separable functions in conjunction with convex underestimators.   
\begin{definition}
\cite{kronqvist2017reformulations} A function $g: \mathbb{R}^n \rightarrow \mathbb{R}$ is convex \textit{partially additively separable}, if it can be written as
$
    g(\boldsymbol{x})= \sum_{i \in \mathcal{M
    }} h_i(\boldsymbol{y}_i), 
$
where each function $h_i$ is convex, and the vectors $\boldsymbol{y}_i$ contain strict subsets of the components in $\boldsymbol{x}$, \ie each $\boldsymbol{y}_i$ are subvectors of $\boldsymbol{x}$ with a strictly smaller number of components than $\boldsymbol{x}$
\end{definition}

\section{Formulations between convex hull and big-M}\label{sec:3}
This section constructs the $P$-split formulations and analyzes the strength of the resulting continuous relaxations. For simplicity, this section considers disjunctions with one constraint per disjunct, \ie  $|\mathcal{C}_{l}| = 1, \ \forall l \in  \mathcal{D}$. Section~\ref{sec:4} extends the results to multiple constraints per disjunct.
The key idea is to partition the variables into $P$ sets, with the corresponding index sets $\mathcal{I}_1, \dots, \mathcal{I}_P$, to lift the disjunction into higher dimensional space in a disaggregated form. Based on the variable partitioning, we introduce $P\times |\mathcal{D}|$ auxiliary variables $\alpha_s^l \in \mathbb{R
 }$ along with the constraints 
\begin{equation}
    \label{eq:alpha_vars}
     \underset{i \in \mathcal{I}_s
     }{\sum}h_{i,l}(x_i) \leq \alpha_s^l \quad \quad \forall s\in \{1, \dots, P\}, \ \forall l \in \mathcal{D}.
\end{equation}
The partitioning must be chosen such that all variables are included in exactly one partitioning, \ie  $\mathcal{I}_1\cup \dots \cup \mathcal{I}_P = \{1, \dots, n\}$ and $\mathcal{I}_i \cap \mathcal{I}_j = \emptyset, i \neq j$. Section~\ref{sec:6} discusses partitioning strategies. If $\sum_{i \in \mathcal{I}_s
     } h_{i,l}(x_i)$ is affine for any $s\in \{1,\dots,P\}$, then \eqref{eq:alpha_vars} can be strengthened to an equality constraint without losing convexity. Auxiliary variables $\alpha_s^l$ and Constraints~\eqref{eq:alpha_vars} allow us to reformulate Disjunction \eqref{eq:main_disjunction} into a disaggregated form
\begin{equation}
\label{eq:main_disjunction_lifted}
\begin{matrix}
\begin{aligned}
    &\underset{l \in \mathcal{D}}{\lor} \begin{bmatrix}  
    g_l(\boldsymbol{x}) \leq b_l\\
      \boldsymbol{x} \in \mathcal{X} 
    \end{bmatrix}\\ 
  
    \end{aligned}
\end{matrix}
\quad \quad 
\longrightarrow  \quad \quad  
\begin{matrix}
\begin{aligned}
    &\underset{l \in \mathcal{D}}{\lor} 
    \begin{bmatrix}
    \begin{aligned}
    &\underset{i \in \mathcal{I}_1}{\sum}h_{i,l}(x_i) \leq \alpha^l_1\\[-0.1cm]
    & \quad \quad \quad \vdots \\
    \vspace{0.1cm}
    &\underset{i \in \mathcal{I}_P}{\sum}h_{i,l}(x_i) \leq \alpha^l_P\\
    \vspace{0.1cm}
    &\sum_{s=1}^P \alpha^l_s  \leq b_l\\
    &\boldsymbol{x} \in \mathcal{X}, \boldsymbol{\alpha}^l \in \mathbb{R}^{P}\ \forall\ l \in \mathcal{D}
    \end{aligned}
    \end{bmatrix} 
    .  
    \end{aligned}
\end{matrix}
\end{equation}
Reformulation~\eqref{eq:main_disjunction_lifted} splits the constraint in the disjunction into $P$ parts. 
 \begin{remark}
 If the same (or scaled) sum of functions ${\sum}_{i \in \mathcal{I}_s}h_{i,l}(x_i)$ factor appears in multiple disjuncts, then it can be constrained by a single $\alpha$ variable across those disjuncts. Introducing unnecessary $\alpha$ variables can produce weaker relaxations: this is an observation familiar in the interval arithmetic and factorable programming literature \cite{misener2014antigone,vigerske2018scip}. Assumption 3 states that we introduce the fewest $\alpha$ variables for the given variable partitioning. 
 \end{remark}
 By Assumption 2 (boundedness), we define bounds on the auxiliary variables  
\begin{equation}
\label{eq:bounds_alpha}
    \begin{aligned}
        & \ubar{\alpha}^l_s := \min_{\boldsymbol{x} \in \mathcal{X}} \underset{i \in \mathcal{I}_s}{\sum}h_{i,l}(x_i), \quad & \bar{\alpha}^l_s := \max_{\boldsymbol{x} \in \mathcal{X}} \underset{i \in \mathcal{I}_s}{\sum}h_{i,l}(x_i) .
    \end{aligned}
\end{equation}
The $P$-split formulations do not require tight bounds, but weak bounds produce weaker relaxations. Note that stronger bounds could be obtained by optimizing over the disjunctive constraint and not over the superset $\mathcal{X}$. We use the bounds as defined in \eqref{eq:bounds_alpha}, as these are in general computationally more tractable. 

The two formulations of the disjunction in~\eqref{eq:main_disjunction_lifted} have the same feasible set in the $\boldsymbol{x}$-variable space. Next, we relax the disjunction by defining all constraints containing the $\boldsymbol{x}$ variables as global constraints. With the auxiliary variable bounds, the disjunction becomes
\begin{equation}
\label{eq:main_disjunction_splitted}
\begin{aligned}
    &\underset{l \in \mathcal{D}}{\lor} \begin{bmatrix}
    \begin{aligned}
    &\sum_{s=1}^P \alpha^l_s  \leq b_l\\
    &\ubar{\alpha}^l_s\leq \alpha^l_s \leq \bar{\alpha}^l_s  \quad \forall s \in \{1,\dots, P\}
    \end{aligned}
    \end{bmatrix} 
    \\
    &\underset{i \in \mathcal{I}_s}{\sum}h_{i,l}(x_i) \leq \alpha^l_s  &&\forall s \in \{1,\dots, P\}, \ \forall \ l \in \mathcal{D} \\
    &\boldsymbol{x} \in \mathcal{X}, \boldsymbol{\alpha}^l \in \mathbb{R}^{P}\ &&\forall \ l \in \mathcal{D}.  
    \end{aligned}
\end{equation}
\begin{definition}
Formulation \eqref{eq:main_disjunction_splitted} is a $P$-split representation of Disjunction \eqref{eq:main_disjunction}.
\end{definition}
Reformulating Disjunction~\eqref{eq:main_disjunction} into the $P$-split representation effectively linearizes the disjunction. Linearizing the disjunction allows us avoid the numerical difficulties related to division by zero arising from forming the convex hull of a general nonlinear disjunction \cite{sawaya2007computational}. The drawback of this linearization is that the continuous relaxation of resulting formulations may not be as tight as possible when projected to the x-space, \ie may not be sharp. 
Section \ref{sec:properties_one} analyzes the strength of the continuous relaxation.

\begin{lemma}\label{lemma_feasset}
The feasible set of $P$-split representation \eqref{eq:main_disjunction_splitted} projected onto the $\boldsymbol{x}$-space is equal to the feasible set of Disjunction~\eqref{eq:main_disjunction}. 
\end{lemma}
\begin{proof}
A variable combination $\bar{\boldsymbol{x}}, \bar{\boldsymbol{\alpha}}^1, \dots,  \bar{\boldsymbol{\alpha}}^{|\mathcal{D}|}$ that is feasible for \eqref{eq:main_disjunction_splitted} must also satisfy $g_l(\bar{\boldsymbol{x}}) \leq b_l$ for at least one $l \in \mathcal{D}$, \ie satisfy the original disjunction in \eqref{eq:main_disjunction_lifted}. Similarly, for an $\Tilde{\boldsymbol{x}}$ satisfying the constraints of the original disjunction in \eqref{eq:main_disjunction_lifted} we can trivially find $\alpha$ variables such that all the constraints in \eqref{eq:main_disjunction_splitted} are satisfied.
\end{proof}

Using the extended formulation \cite{balas1998disjunctive} to represent the convex hull of the disjunction in \eqref{eq:main_disjunction_splitted} results in the \textit{$P$-split formulation}

\allowdisplaybreaks{
\begin{equation}
\label{eq:p-split}
\tag{$P$-split}
\begin{aligned}
& \alpha^l_s = \underset{d \in \mathcal{D}}{\sum} \nu^{\alpha^l_s}_d && \forall \ s \in \{1, \dots, P\}, \ \forall\ l \in \mathcal{D} \\
& \sum_{s=1}^P \nu^{\alpha^l_s}_l  \leq b_l\lambda_l &&\forall\ l \in \mathcal{D}\\
& \ubar{\alpha}^l_s\lambda_d \leq \nu^{\alpha^l_s}_d \leq \bar{\alpha}^l_s\lambda_d  &&\forall \ s \in \{1, \dots, P\},\forall\ l,d \in \mathcal{D}\\
 &\underset{i \in \mathcal{I}_s}{\sum}h_{i,l}(x_i) \leq \alpha^l_s  &&\forall\ s \in \{1,\dots, P\}, \ \forall \ l \in \mathcal{D} \\
 & \underset{l \in \mathcal{D}}{\sum}\lambda_l = 1, \quad   \boldsymbol{\lambda} \in \{0, 1\}^{|\mathcal{D}|}\\
&\boldsymbol{x} \in \mathcal{X}, \boldsymbol{\alpha}^l \in \mathbb{R}^{P},  \ \boldsymbol{\nu}^{\alpha^l_s} \in \mathbb{R}^{|\mathcal{D}|} &&\forall\ s \in \{1,\dots, P\}, \ \forall \ l \in \mathcal{D}\ ,
\end{aligned}
\end{equation}}

which form convex MIP constraints.  To clarify our terminology, $P$ can be assigned an integer value: a 3-split formulation is a formulation \eqref{eq:p-split} where the constraints of the original disjunction are split into three parts ($P = 3$). Assuming the original disjunction is part of a larger optimization problem that may contain multiple disjunctions, we enforce integrality of the $\lambda$-variables even if we obtain the convex hull of the disjunction. Proposition \ref{prop_feasset} shows the correctness of the \eqref{eq:p-split} formulation of the original disjunction.

\begin{proposition}\label{prop_feasset}
The feasible set projected onto the $\boldsymbol{x}$-space of formulation \eqref{eq:p-split} is equal to the feasible set of the original  disjunction in
~\eqref{eq:main_disjunction_lifted}.
\end{proposition}
\begin{proof}
Omitted as it trivially follows from Lemma \ref{lemma_feasset} and basic properties of the extended convex hull formulation \cite{balas1998disjunctive}.
\end{proof}
Proposition \ref{prop_feasset} states that the \eqref{eq:p-split} formulation is correct for integer feasible solutions, but it does not provide any insight on the strength of its continuous relaxation. The following subsections further analyze the properties of the \eqref{eq:p-split} formulation and its relation to the big-M and convex hull formulations. 

\begin{remark}
If each $\alpha$ variable appears only in one disjunct, the \eqref{eq:p-split} formulation introduces $P\cdot \left(|\mathcal{D}|^2 +|\mathcal{D}| \right)$ continuous  and $|\mathcal{D}|$ binary variables.  Unlike the extended convex hull formulation, the number of \say{auxiliary} variables is independent of $n$, \ie the number of variables in the original disjunction. Section \ref{sec:7} presents applications where $|\mathcal{D}| << n$ and the \eqref{eq:p-split} formulation is significantly smaller than the extended convex hull formulation. For disjunctions with general nonlinearities,  \eqref{eq:p-split} formulations may be simpler as they avoid the perspective function.
\end{remark}

\subsubsection{Illustrative example}
Before rigorously analyzing properties of P-split formulations, we give a simple illustrative example. 
To show the difference in relaxation strength between P-split formulations with different numbers of splits, consider the disjunctive feasible set
\begin{equation}
\label{eq:example1}
\tag{ex-1}
\begin{aligned}
    &\begin{bmatrix}
    \sum_{i=1}^4 x_i ^2  \vspace{0.1cm} \leq 1\\
    \boldsymbol{x} \in [-4,\ 4]^4
    \end{bmatrix}
    \
    \lor \
    \begin{bmatrix}
    \sum_{i=1}^4 -x_i \leq -12 \vspace{0.1cm} \\
    \boldsymbol{x} \in [-4,\ 4]^4
    
    \end{bmatrix}.   
\end{aligned}
\end{equation}
For the disjunction, we form 1-split, 2-split, and 4-split formulations, and plot their continuous relaxations. For the 2-split formulation, we use the variable partitioning $\{x_1, x_2\}$ and $\{x_3, x_4\}$. With 1- and 4-split formulations, the partitioning is unique. We use the tightest globally valid bounds for all the split variables (see Appendix~A).
Figure~\ref{fig:relaxations} shows the feasible sets of the continuous relaxations of the resulting $P$-split formulations projected onto $x_1, x_2$. 
Figure \ref{fig:relaxations_independent_ap} in Appendix~A plots the relaxations with weaker bounds and shows that the advantage over big-M remains. More details on different formulations are given in Appendix~A.

\begin{figure}[t]
    \centering
    \begin{subfigure}[t]{.32\linewidth}
    \includegraphics[trim={1.5cm 0cm 2cm 0cm},clip,width=.99\linewidth]{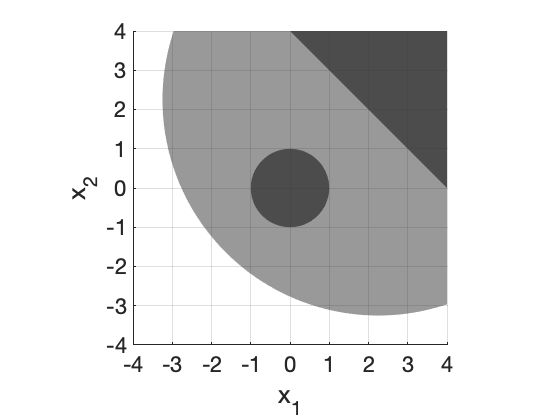}
    \caption{\centering 1-split/big-M $\left(\{x_1, x_2, x_3, x_4\}\right)$}
    \end{subfigure}
    \begin{subfigure}[t]{.32\linewidth}
    \includegraphics[trim={1.5cm 0cm 2cm 0cm},clip,width=.99\linewidth]{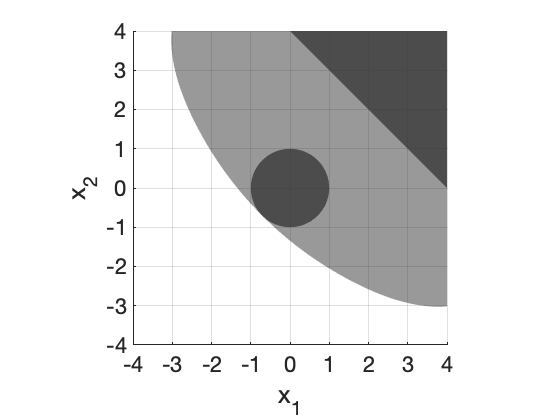}
    \caption{2-split \\ $\left(\{x_1, x_2\}, \{x_3, x_4\}\right)$}
    \end{subfigure}
    \begin{subfigure}[t]{.32\linewidth}
    \includegraphics[trim={1.5cm 0cm 2cm 0cm},clip,width=.99\linewidth]{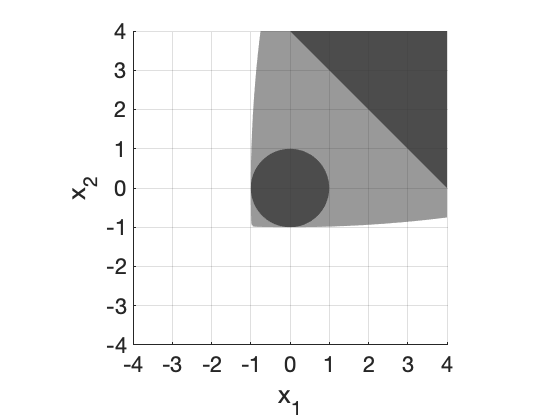}
    \caption{\centering 4-split $\left(\{x_1\},\{x_2\},\{x_3\},\{x_4\}\right)$}
    \end{subfigure}
      \caption{The dark regions show the feasible set of \eqref{eq:example1} in the $x_1,x_2$ space. The light grey areas show the continuously relaxed feasible set of the P-split formulations, and the sets in the parentheses show the partitioning of variables.}
    \label{fig:relaxations}
\end{figure}

 Increasing the number of splits for (\ref{eq:example1}) clearly leads to volumetrically tighter relaxations.  Since we used the tightest valid bounds, the 1-split is equivalent to the tightest big-M formulation with individual big-M coefficients.  Interestingly,  Figure~\ref{fig:relaxations} also shows that the 4-split is not strictly tighter than the 2-split everywhere. Next, we analyze the relaxation strengths of the different P-split formulations and show under which circumstances the formulation becomes strictly tighter.

\subsection{Properties of $P$-split formulations and their continuous relaxations}\label{sec:properties_one}

This section studies the continuous relaxations of $P$-split formulations, and how they compare to convex hull and big-M formulations. Theorem \ref{thm:bigM} shows that continuous relaxations of the $1$-split formulation and the big-M formulation have the same feasible set in the $\boldsymbol{x}$ space. 
Theorem \ref{thm:p_1} proves that introducing an additional split, \ie $(P+1)$-split, results in a formulation with an equally tight or tighter continuous relaxation than that of the $P$-split. Section~4 shows that, in some circumstances, we can construct the convex hull of a disjunction with a $P$-split formulation. This section only considers one constraint per disjunct; Section~4 addresses multiple constraints per disjunct.

\begin{theorem}\label{thm:bigM}
The 1-split formulation is equivalent to the big-M formulation in the sense that their continuous relaxations have the same set of feasible $x$ variables.
\end{theorem}
\begin{proof}
The first disaggregated variables $\nu^{\alpha^l}_1$ can directly be eliminated by substituting in  $\nu^{\alpha^l}_1 = \alpha^l - \underset{d \in \mathcal{D}_k \setminus 1}{\sum} \nu^{\alpha^l}_d$. We project out the remaining disaggregated variables $\nu^{\alpha^l}_d$ from the 1-split formulation using Fourier-Motzkin elimination. 
We eliminate trivially redundant constraints, \eg  $\ubar{\alpha}^l\lambda_d \leq \bar{\alpha}^l\lambda_d$, resulting in
\begin{equation}
\label{eq:1-split_bigM}
\begin{aligned}
&  \alpha^l  \leq b_l\lambda_l +  \underset{d \in \mathcal{D} \setminus l}{\sum}\bar{\alpha}^l\lambda_d  && \forall l \in \mathcal{D} \\
  &\sum_{i=1}^nh_{i,l}(x_i) \leq \alpha^l  &&\forall \ l \in \mathcal{D} \\
 & \underset{l \in \mathcal{D}}{\sum}\lambda_l = 1, \quad  \boldsymbol{\lambda} \in \{0, 1\}^{|\mathcal{D}|},\boldsymbol{x} \in \mathcal{X}, \boldsymbol{\alpha}^l \in \mathbb{R} \ &&\forall \ l \in \mathcal{D}.
\end{aligned}
\end{equation}
The auxiliary variables $ \alpha^l$ can be eliminated by combining the first and second constraints in \eqref{eq:1-split_bigM}. Note that the smallest valid big-M coefficients are $M^l = \bar{\alpha}^l - b_l $, and by introducing this notation we can write \eqref{eq:1-split_bigM} as
\begin{equation}
\begin{aligned}
&  \sum_{i=1}^nh_{i,l}(x_i)    \leq b_l  + M^l(1- \lambda_l)  && \forall l \in \mathcal{D}_k \\
 & \underset{l \in \mathcal{D}}{\sum}\lambda_l = 1, \quad  \boldsymbol{\lambda} \in \{0, 1\}^{|\mathcal{D}|},\ \boldsymbol{x} \in \mathcal{X}.
\end{aligned}
\end{equation}
Thus, projecting out the auxiliary and disaggregated variables in the $1$-split formulation leaves us with the big-M formulation. 
\end{proof}
By Theorem~1, the 1-split formulation is an alternative form of the classical big-M formulation. But the 1-split formulation introduces $|\mathcal{D}|^2 + |\mathcal{D}|$ new variables, so there is no advantage of the 1-split formulation compared to the big-M formulation.

To obtain a strict relation between different $P$-split relaxations, we assume that bounds on the auxiliary variables are \textit{additive}, as discussed in Definition 3.

\begin{definition}
We say that the upper and lower bounds for the functions $h_i$, originating from the constraint $\sum_{i=1}^n h_{i}(x_i) \leq 0$, are additive over $\mathcal{X}$ if 
\begin{equation}
\label{eq:independent}
\begin{aligned}
  &\min_{\boldsymbol{x} \in \mathcal{X}} \left(\underset{i \in \mathcal{I}_{s_1}
     }{\sum}h_{i}(x_i) +\underset{i \in \mathcal{I}_{s_2}
     }{\sum}h_{i}(x_i) \right) =  \min_{\boldsymbol{x} \in \mathcal{X}} \ \underset{i \in \mathcal{I}_{s_1}
     }{\sum}h_{i}(x_i) +\min_{\boldsymbol{x} \in \mathcal{X}} \ \underset{i \in \mathcal{I}_{s_2}
     }{\sum}h_{i}(x_i) \\
  &\max_{\boldsymbol{x} \in \mathcal{X}} \left(\underset{i \in \mathcal{I}_{s_1}
     }{\sum}h_{i}(x_i) +\underset{i \in \mathcal{I}_{s_2}
     }{\sum}h_{i}(x_i) \right) =   \max_{\boldsymbol{x} \in \mathcal{X}} \ \underset{i \in \mathcal{I}_{s_1}
     }{\sum}h_{i}(x_i) +\max_{\boldsymbol{x} \in \mathcal{X}} \ \underset{i \in \mathcal{I}_{s_2}
     }{\sum}h_{i}(x_i) ,
\end{aligned}
\end{equation}
hold for all valid split partitionings, \ie $ \forall \ \mathcal{I}_{s_1}, \mathcal{I}_{s_2} \subset \{1, \dots, n\}\ | \ \mathcal{I}_{s_1} \cap \mathcal{I}_{s_2} = \emptyset $.

\end{definition}
The simplest case with additive bounds is when $\mathcal{X}$ is an $n$-dimensional box, or when the bounds are calculated with interval arithmetic. 
Note we only use the additive bound property to derive a strict relation between different $P$-split formulations: it is not required by the $P$-split formulations. Tighter valid bounds result in an overall stronger continuous relaxation. 
Assuming that the bounds $\ubar{\alpha}^l_s$ and  $\bar{\alpha}^l_s$ are additive, Theorem \ref{thm:p_1} shows that increasing the number of splits $P$ results in a formulation with an equally tight or tighter continuous relaxation.

\begin{theorem}\label{thm:p_1}
For a disjunction with additive bounds, a $(P+1)$-split formulation, obtained by splitting one of the variable groups in a $P$-split formulation, is always as tight or tighter than the corresponding P-split formulation. 
\end{theorem}
\begin{proof}
W.l.o.g., we assume that the $(P+1)$-split formulation is obtained by splitting the first group of variables $\mathcal{I}_1$ in the $P$-split formulation into subsets $\mathcal{I}_a$ and $\mathcal{I}_b$, \ie, $\mathcal{I}_a \cup \mathcal{I}_b = \mathcal{I}_1$ and $\mathcal{I}_a \cap \mathcal{I}_b = \emptyset$. The $(P+1)$-formulation can be written
\begin{align}
\label{eq:p+1,1}
& \alpha^l_s = \underset{d \in \mathcal{D}}{\sum} \nu^{\alpha^l_s}_d && \forall \ s \in \{a,b,2,3, \dots, P\}, \ \forall\ l \in \mathcal{D} \\
\label{eq:p+1,2}
& \nu^{\alpha^l_a}_l + \nu^{\alpha^l_b}_l+ \sum_{s=2}^P \nu^{\alpha^l_s}_l  \leq b_l\lambda_l &&\forall\ l \in \mathcal{D}\\
\label{eq:p+1,4}
& \ubar{\alpha}^l_s\lambda_d \leq \nu^{\alpha^l_s}_d \leq \bar{\alpha}^l_s\lambda_d  &&\forall \ s \in \{a,b,2,3, \dots, P\},\forall\ l,d \in \mathcal{D}\\
\label{eq:p+1,5}
&\underset{i \in \mathcal{I}_s}{\sum}h_{i,l}(x_i) \leq \alpha^l_s  &&\forall\ s \in \{a,b,2,3,\dots, P\}, \ \forall \ l \in \mathcal{D} \\
 \label{eq:p+1,7}
& \underset{l \in \mathcal{D}}{\sum}\lambda_l = 1, \quad   \boldsymbol{\lambda} \in \{0, 1\}^{|\mathcal{D}|}\\
 \label{eq:p+1,9}
&\boldsymbol{x} \in \mathcal{X}, \boldsymbol{\alpha}^l \in \mathbb{R}^{P+1},  \ \boldsymbol{\nu}^{\alpha^l_s} \in \mathbb{R}^{|\mathcal{D}|} &&\forall\ s \in \{a, b, 2, 3, \dots, P\}, \ \forall \ l \in \mathcal{D}.
\end{align}
Since the bounds are additive (Definition 3), we have $ \ubar{\alpha}^l_a + \ubar{\alpha}^l_b = \ubar{\alpha}^l_1$ and $\bar{\alpha}^l_a + \bar{\alpha}^l_b = \bar{\alpha}^l_1$. Aggregating the first two constraints in \eqref{eq:p+1,4}, \ie $s = \{a,b\}$, results in
\begin{align}
\label{eq:agg1}
& \ubar{\alpha}^l_{1}\lambda_d \leq \nu^{\alpha^l_{a}}_d + \nu^{\alpha^l_{b}}_d \leq \bar{\alpha}^l_{1}\lambda_d  &&\forall\ l,d \in \mathcal{D}.    
\end{align}
Replacing the two summed constraints with \eqref{eq:agg1} results in a relaxation.
Similarly, we can aggregate the first two constraints in \eqref{eq:p+1,1} and \eqref{eq:p+1,5}, respectively, into
\begin{align}
\label{eq:agg2}
    & \alpha^l_{a} +\alpha^l_{b} = \underset{d \in \mathcal{D}}{\sum} \left(\nu^{\alpha^l_{a}}_d +\nu^{\alpha^l_{b}}_d\right) && \forall\ l \in \mathcal{D} \\
\label{eq:agg3}
     &\underset{i \in \mathcal{I}_{a}}{\sum}h_{i,l}(x_i) +\underset{i \in \mathcal{I}_{b}}{\sum}h_{i,l}(x_i) \leq \alpha^l_{a} + \alpha^l_{b}  &&\forall \ l \in \mathcal{D},
\end{align}
which again relaxes the constraints. We can now relax the $(P+1)$-split formulation, by using the aggregated constraints in \eqref{eq:agg1}, \eqref{eq:agg2}, and \eqref{eq:agg3}. Furthermore, if we substitute $\alpha^l_{1} = \alpha^l_{a} +\alpha^l_{b} \ \forall \ l \in \mathcal{D}$ and $\nu^{\alpha^l_{1}}_d = \nu^{\alpha^l_{a}}_d + \nu^{\alpha^l_{b}}_d \ \forall \ l,d \in \mathcal{D}$, then we recover the constraints of the original $P$-split formulation corresponding to $\mathcal{I}_1$. 
Thus, by relaxing the $(P+1)$-formulation through constraint aggregation, we obtain the original $P$-split formulation.   
\end{proof}
The proof of Theorem 2 shows that the main difference between the $(P+1)$-split and $P$-split formulations is the disaggregation of the constraints in \eqref{eq:agg1}, \eqref{eq:agg2}, and \eqref{eq:agg3}. Corollary \ref{cor:tighter} shows that the $(P+1)$-split formulation can be strictly tighter.

\begin{corollary}\label{cor:tighter}
 A $(P+1)$-split formulation, obtained by splitting one of the variable groups in a $P$-split formulation can be strictly tighter.
\end{corollary}
\begin{proof}
Any variable combination satisfying constraints \eqref{eq:p+1,2}, \eqref{eq:p+1,5}, and \eqref{eq:p+1,7} always satisfies \eqref{eq:agg1}, \eqref{eq:agg2}, and \eqref{eq:agg3}, but not the other way around. For example, we may have $\underset{i \in \mathcal{I}_{a}}{\sum}h_{i,l}(x_i) +\underset{i \in \mathcal{I}_{b}}{\sum}h_{i,l}(x_i) \leq \alpha^l_{a} + \alpha^l_{b}$, and $\underset{i \in \mathcal{I}_{a}}{\sum}h_{i,l}(x_i) > \alpha^l_{a}$. 
\end{proof}
For the illustrative example (\ref{eq:example1}), the tightest valid bounds for the $\alpha$-variables are not additive (see Appendix A). Figure~\ref{fig:relaxations} illustrates that due to the non-additive bounds, the relaxations do not form a hierarchy, \ie the  continuously relaxed 2-split (4-split) formulation is not strictly tighter than the $1$-split (2-split). However, for (\ref{eq:example1}), increasing the number of splits leads to volumetrically tighter relaxations.  
Appendix A continues on (\ref{eq:example1}) using additive bounds and illustrates that increasing $P$ results in a strictly tighter formulation.

Note that Theorem~\ref{thm:p_1} and Corallary~\ref{cor:tighter} build on the fact that a given P-split formulation is split one step further. Different variable partitionings for a given number of splits can result in weaker and stronger relaxations. Partitioning strategies are further discussed in Section~\ref{sec:6}.

\subsubsection{Special case: 2-term disjunctions}
So far, we have considered disjunctions with an arbitrary number of disjuncts. 
But \textit{2-term disjunctions}, or disjunctions with two disjuncts, are
of special interest. 
Indicator constraints, which either hold or are relaxed depending on a binary variable,
can be modeled as a 2-term disjunction \cite{belotti2016handling,bonami2015mathematical}. 
Indicator constraints arise in classification~\cite{brooks2011support} and chance constraints \cite{luedtke2010integer}. For a 2-term disjunction, Theorem~\ref{thm:theorem3} projects out the auxiliary variables to obtain a non-extended formulation.  

\begin{theorem}\label{thm:theorem3}
For a  two-term disjunction, the $P$-split formulation has the following non-extended realization
\begin{equation}
\begin{aligned}
\label{eq:P-2-split-nonlifted}
& \underset{j \in \mathcal{S}_p}{\sum}\left( \underset{i \in \mathcal{I}_j}{\sum}h_{i,1}(x_i) \right)\leq \left(b_1 - \underset{s \in \mathcal{S}\setminus \mathcal{S}_p}{\sum} \ubar{\alpha}^1_s\right)\lambda_1 + \underset{s \in \mathcal{S}_p}{\sum}\bar{\alpha}^1_s\lambda_2 && \forall \mathcal{S}_p \subset \mathcal{S}\\
& \underset{j \in \mathcal{S}_p}{\sum}\left( \underset{i \in \mathcal{I}_j}{\sum}h_{i,2}(x_i) \right)\leq \left(b_2 - \underset{s \in \mathcal{S}\setminus \mathcal{S}_p}{\sum} \ubar{\alpha}^2_s\right)\lambda_2 + \underset{s \in \mathcal{S}_p}{\sum}\bar{\alpha}^2_s\lambda_1 && \forall \mathcal{S}_p \subset \mathcal{S}\\
 & \lambda_1 + \lambda_2  = 1, \ \ \boldsymbol{\lambda} \in \{0, 1\}^2, \ \ \boldsymbol{x} \in \mathcal{X},
\end{aligned}
\end{equation}
where $\mathcal{S} = \{1, 2, \dots P\}$ and the sets $\mathcal{I}_j$ contain the indices for the variable partitionings for the $P$-split formulation \cite{kronqvist2021between}.
\end{theorem}
\begin{proof}
We prove the theorem by projecting out all the auxiliary variables from \eqref{eq:p-split}. First, we use the equality constraints for the disaggregated variables ($\alpha_s^l = \nu_1^{\alpha_s^l} + \nu_2^{\alpha_s^l}$) to eliminate the variables $\nu^{\alpha^l_s}_1$, resulting in
\begin{align}
& \sum_{s=1}^P\left( \alpha^1_s - \nu^{\alpha^1_s}_2\right)  \leq b_1\lambda_1 \label{eq:p-2-s-eq1}\\
& \sum_{s=1}^P \nu^{\alpha^2_s}_2  \leq b_2\lambda_2 \label{eq:p-2-s-eq2}\\
& \ubar{\alpha}^l_s\lambda_1 \leq \alpha^l_s - \nu^{\alpha^l_s}_2 \leq \bar{\alpha}^l_s\lambda_1  &&\forall s  \in \{1, 2 , \dots, P\},\forall\ l \in  \{1, 2\}\label{eq:p-2-s-eq3}\\
& \ubar{\alpha}^l_s\lambda_2 \leq \nu^{\alpha^l_s}_2 \leq \bar{\alpha}^l_s\lambda_2  &&\forall s  \in \{1, 2 , \dots, P\},\forall\ l \in  \{1, 2\}\label{eq:p-2-s-eq4}\\
 &\underset{i \in \mathcal{I}_s}{\sum}h_{i,l}(x_i) \leq \alpha^l_s  &&\forall\ s \in \{1, 2 , \dots, P\}, \ \forall \ l \in  \{1, 2\} \label{eq:p-2-s-eq5}\\
 & \lambda_1 + \lambda_2  = 1, \quad  \boldsymbol{\lambda} \in \{0, 1\}^2 \label{eq:p-2-s-eq6}\\
&\boldsymbol{x} \in \mathcal{X}, \boldsymbol{\alpha}^l \in \mathbb{R}^{P}, \boldsymbol{\nu}^{\alpha^l_s} \in \mathbb{R}^P\ &&\forall \ l \in \{1, 2\}, \forall\ s \in \{1, 2 , \dots, P\}.    
\end{align}
By Fourier-Motzkin elimination, we can also project out the $\nu^{\alpha^1_s}_2$ variables. Combining the constraints in \eqref{eq:p-2-s-eq3} and \eqref{eq:p-2-s-eq4} only results in trivially redundant constraints, \eg  $\alpha^l_s \leq \bar{\alpha}^l_s(\lambda_1 + \lambda_2)$. We eliminate the first variable $\nu^{\alpha^1_1}_2$ by combining \eqref{eq:p-2-s-eq1} with \eqref{eq:p-2-s-eq3}--\eqref{eq:p-2-s-eq4} which results in the two new constraints
\begin{align}
\tag{\ref{eq:p-2-s-eq1}a}
\label{eq:motzkin1}
& \sum_{s=2}^P\left( \alpha^1_s - \nu^{\alpha^1_s}_2\right)  \leq  b_1\lambda_1 - \ubar{\alpha}^1_1\lambda_1,\\
\tag{\ref{eq:p-2-s-eq1}b}
\label{eq:motzkin2}
& \sum_{s=2}^P\left( \alpha^1_s - \nu^{\alpha^1_s}_2\right) + \alpha^1_1  \leq  b_1\lambda_1 + \bar{\alpha}^1_1\lambda_2.
\end{align}
Eliminating the next variable is done by repeating the procedure of combining the new constraints \eqref{eq:motzkin1}--\eqref{eq:motzkin2} with the corresponding inequalities in \eqref{eq:p-2-s-eq3}--\eqref{eq:p-2-s-eq4}. Each elimination step doubles the number of constraints originating from inequality \eqref{eq:p-2-s-eq1}. Eliminating all the variables $\nu^{\alpha^1_s}_2$ results in the set of constraints  
\begin{equation}
\label{eq:motzkin3}
   \underset{s \in \mathcal{S}_p}\sum \alpha^1_s \leq \left(b_1 - \underset{s \in \mathcal{S}\setminus \mathcal{S}_p}{\sum} \ubar{\alpha}^1_s\right)\lambda_1 + \underset{s \in \mathcal{S}_p}{\sum}\bar{\alpha}^1_s\lambda_2 \quad \forall \mathcal{S}_p \subset \mathcal{S}.
\end{equation}
The $\alpha_s^1$ variables can be eliminated by combining \eqref{eq:p-2-s-eq5} and \eqref{eq:motzkin3}, leaving us with the first set of constraints in \eqref{eq:P-2-split-nonlifted}.  
The variables $\nu^{\alpha^2_s}_2$ and $\alpha^2_s$ are eliminated by same steps, resulting in the second set of constraints in \eqref{eq:P-2-split-nonlifted}. 
\end{proof}

Corollary \ref{cor:splitting} shows how Theorem~\ref{thm:theorem3} gives insight into relaxation strength. Corollary \ref{cor:splitting} is covered by Theorem~\ref{thm:p_1}, but the proof offers an alternative interpretation.

\begin{corollary}\label{cor:splitting}
For a two-term disjunction with additive bounds, the feasible set of the continuous relaxation of any $(P+1)$-split formulation, formed by splitting one variable group in a $P$-split formulation, is contained within the feasible set of continuous relaxation of the original $P$-split formulation. 
\end{corollary}
\begin{proof}
Theorem~\ref{thm:theorem3} shows that all the constraints in the  non-extended realization of a $P$-split formulation are contained in the $(P+1)$-split formulation, assuming additive bounds. The non-extended $(P+1)$-split formulation contains some additional valid inequalities that are not in the $P$-split formulation. 
\end{proof}
Theorem~\ref{thm:theorem3} realizes any $P$-split formulation for a two-term disjunction in the same variable space as the big-M formulation, and the constraints could be implemented as cutting planes. But the number of constraints in \eqref{eq:P-2-split-nonlifted} grows exponentially with the number of splits, so an efficient cutting plane implementation would require an efficient cut-selection technique. Such techniques are not considered in this paper.

\section{Disjuncts with multiple constraints}\label{sec:4}
This section studies the general case of Disjunction \eqref{eq:main_disjunction} with multiple constraints per disjunct. By Assumption~4, there are far fewer constraints than variables in each disjunct. Without this structure, the $P$-split formulations are inadvisable: the number of new auxiliary variables may exceed the number of variables in a convex hull formulation.
By treating each constraint of each disjunct individually, we can split each constraint by introducing a unique set of $\alpha$-variables. We thus need a triple index $\alpha_s^{l,k}$, where $k$ indexes the different constraints of the disjuncts. By the same procedure used for the single constraint case, we obtain
\allowdisplaybreaks{
\begin{align*}
& \alpha^{l,k}_s = \underset{d \in \mathcal{D}}{\sum} \nu^{\alpha^{l,k}_s}_d && \forall \ s \in \{1, \dots, P\}, \ \forall\ l \in \mathcal{D}, \ \forall k \in \mathcal{C}_l \\
& \sum_{s=1}^P \nu^{\alpha^{l,k}_s}_l  \leq b_{l,k}\lambda_l &&\forall\ l \in \mathcal{D}, \ \forall k \in \mathcal{C}_l\\
& \ubar{\alpha}^{l,k}_s\lambda_d \leq \nu^{\alpha^{l,k}_s}_d \leq \bar{\alpha}^{l,k}_s\lambda_d  &&\forall \ s \in \{1, \dots, P\},\forall\ l,d \in \mathcal{D}, \ \forall k \in \mathcal{C}_l\\
\label{eq:p-split-multicon}
\tag{$P$-split$^*$}
 &\underset{i \in \mathcal{I}_s}{\sum}h_{i,k}(x_i) \leq \alpha^{l,k}_s  &&\forall\ s \in \{1,\dots, P\}, \ \forall \ l \in \mathcal{D}, \ \forall k \in\mathcal{C}_l \\
 & \underset{l \in \mathcal{D}}{\sum}\lambda_l = 1, \quad   \boldsymbol{\lambda} \in \{0, 1\}^{|\mathcal{D}|}\\
&\boldsymbol{x} \in \mathcal{X}, \boldsymbol{\alpha}^{l,k} \in \mathbb{R}^{P},  \ \boldsymbol{\nu}^{\alpha^{l,k}_s} \in \mathbb{R}^{|\mathcal{D}|} &&\forall\ s \in \{1,\dots, P\}, \ \forall \ l \in \mathcal{D}, \ \forall k \in \mathcal{C}_l.
\end{align*}
}

Proposition~1, and the intermediate results leading up to Proposition~1, trivially extend to multiple constraints per disjunct. Therefore, formulation \eqref{eq:p-split-multicon} gives an exact representation of the original disjunction \eqref{eq:main_disjunction} for integer feasible solutions. Theorem \ref{thm:p_split_mult} formalizes two of the main properties of formulation \eqref{eq:p-split-multicon}

\begin{theorem}\label{thm:p_split_mult}
The 1-split formulation in \eqref{eq:p-split-multicon} is equivalent to the big-M formulation. Further splitting one group of variables in a $P$-split formulation, to form a $(P+1)$-split formulation, results in an as tight or tighter continuous relaxation if the bounds on the split variables are additive (see Definition 3). 
\end{theorem}
\begin{proof}

Note \eqref{eq:p-split-multicon} has the same structure as \eqref{eq:p-split}, except that there are individual sets of $\alpha$-variables and constraints originating from the multiple constraints of the disjuncts. As in Theorem~1, we perform Fourier-Motzkin elimination to project out all of the auxiliary variables $\alpha_s^{l,k}$ and disaggregated variables $\nu_d^{\alpha_s^{l,k}}$, obtaining

\begin{equation}
\begin{aligned}
&  \sum_{i=1}^nh_{i,k}(x_i)    \leq b_{l,k}  + M^{l,k}(1- \lambda_l)  && \forall k\in\mathcal{C}_l, \ \forall l \in \mathcal{D}_k \\
 & \underset{l \in \mathcal{D}}{\sum}\lambda_l = 1, \quad  \boldsymbol{\lambda} \in \{0, 1\}^{|\mathcal{D}|},\ \boldsymbol{x} \in \mathcal{X}.
\end{aligned}
\end{equation}

As in Theorem~2, we relax the $(P+1)$-formulation by aggregating the constraints of two of the split variables to obtain the $P$-formulation.
\end{proof}

The main idea behind the $P$-split formulations is to obtain compact approximations of the convex hull with a stronger relaxation than big-M. However, it is interesting to understand the circumstances under which we can recover the true convex hull of the disjunction by a $P$-split formulation. Theorem \ref{thm:sufficient} gives sufficient conditions for a $P$-split formulation to form the convex hull.

\begin{theorem}\label{thm:sufficient}
If all $h_{i,k}$ are affine and $\mathcal{X}$ is a bounded box defined by upper and lower bounds on all variables, a continuous relaxation of the $n$-split formulation (all constraints are split for each variable) results in the convex hull of the disjunction.
\end{theorem}
\begin{proof}
Here, the original disjunction can be written as
\begin{equation}
\label{eq:lin_disjunct}
    \begin{matrix}
\begin{aligned}
    &\underset{l \in \mathcal{D}}{\lor} \begin{bmatrix}  
    \mathbf{A}^l \boldsymbol{x} \leq \boldsymbol{b}_l\\
    \end{bmatrix},
    \end{aligned}
\end{matrix}
\end{equation}
where each disjunct is bounded. W.l.o.g., no column in all matrices $\mathbf{A}^l$ comprises all zeros. As we only have linear constraints in the disjuncts, we never need to introduce more than $n$ auxiliary variables, see Remark 1. Furthermore, the mapping between the $\boldsymbol{\alpha}$ and $\boldsymbol{x}$ variables is given by a linear transformation. The $n$-split representation (\ie $P=n$) can then be written as
\begin{equation}
\label{eq:p-split-rep-lin}
    \begin{matrix}
\begin{aligned}
    &\underset{l \in \mathcal{D}}{\lor} \begin{bmatrix}  
    \mathbf{H}^l\boldsymbol{\alpha} \leq \boldsymbol{b}_l\\
    \end{bmatrix}\\ 
    &\boldsymbol{\alpha} = \mathbf{\Gamma} \boldsymbol{x}, {\boldsymbol{\alpha}} \in \mathbb{R}^{n},
    \end{aligned}
\end{matrix}
\end{equation}
where $\mathbf{H}^l\boldsymbol{\alpha} \leq \boldsymbol{b}_l$ includes the bound constraints $\ubar{\boldsymbol{\alpha}}\leq \boldsymbol{\alpha} \leq \bar{\boldsymbol{\alpha}}$. 
In \eqref{eq:p-split-rep-lin} the constraint $\boldsymbol{x} \in \mathcal{X}$ is redundant as the bound constraints on the $\boldsymbol{x}$ variables are perfectly captured by the bound constraints on the $\boldsymbol{\alpha}$ variables through $\boldsymbol{\alpha} = \mathbf{\Gamma} \boldsymbol{x}$.
The $n$-split formulation is the extended convex hull formulation of \eqref{eq:p-split-rep-lin}. The matrix $\mathbf{\Gamma}$ is invertible, as it can be diagonalized by applying a permutation (reordering the $\boldsymbol{\alpha}$-variables).  
Therefore, we have a bijective mapping between the $\boldsymbol{x}$ and ${\boldsymbol{\alpha}}$ variable spaces (only true for a full $n$-split). The linear transformations preserve an exact representation of the feasible sets, \ie
\begin{equation}
\begin{aligned}
  &  \mathbf{H}^l{\boldsymbol{\alpha}} \leq {\boldsymbol{b}_l} \iff \mathbf{A}^l \mathbf{\Gamma}^{-1}{\boldsymbol{\alpha}} \leq {\boldsymbol{b}_l},  \quad \quad \mathbf{A}^l\boldsymbol{x} \leq {\boldsymbol{b}_l}  \iff \mathbf{H}^l\mathbf{\Gamma} \boldsymbol{x} \leq \boldsymbol{b}_l.    
\end{aligned}
\end{equation}
For any point $\boldsymbol{z}$ in the convex hull of  \eqref{eq:p-split-rep-lin} $\exists\ \tilde{\boldsymbol{\alpha}}^1, \tilde{\boldsymbol{\alpha}}^2, \dots \tilde{\boldsymbol{\alpha}}^{\mathcal{|D}|}$ and $\boldsymbol{\lambda} \in \mathbb{R}_+^{|\mathcal{D}|}:$
\begin{align}
\label{eq:convex-comb}
    &\boldsymbol{z} = \sum_{l=1}^{|\mathcal{D}|}\lambda_l\tilde{\boldsymbol{\alpha}}^l, \quad \sum_{l=1}^{|\mathcal{D}|}\lambda_l = 1,\ \  \mathbf{H}^l\tilde{\boldsymbol{\alpha}}^l \leq {\boldsymbol{b}_l} \quad  \forall\ l \in \mathcal{D}.
\end{align}
Applying the reverse mapping $\mathbf{\Gamma}^{-1}$ to the convex combination of points in \eqref{eq:convex-comb} gives
\begin{equation}
     \mathbf{\Gamma^{-1}}\boldsymbol{z} = \sum_{l=1}^{|\mathcal{D}|}\lambda_l\mathbf{\Gamma^{-1}}\tilde{\boldsymbol{\alpha}}^l.
\end{equation}
By construction, $ \mathbf{A}^l\mathbf{\Gamma^{-1}}\tilde{\boldsymbol{\alpha}}^l \leq \boldsymbol{b}_l \quad  \forall l \in \mathcal{D}$.  The point $\mathbf{\Gamma^{-1}}\boldsymbol{z}$ is a convex combination of points that all satisfy the constraints of one of the disjuncts in \eqref{eq:lin_disjunct}, \ie it belongs to the convex hull of \eqref{eq:lin_disjunct}. The same technique shows that any point in the convex hull of disjunction~\eqref{eq:lin_disjunct} also belongs to that of disjunction~\eqref{eq:p-split-rep-lin}.
\end{proof}

For a disjunction with nonlinear constraints and $P$-split formulations with $P<n$, the proof breaks, as the mapping between the $\boldsymbol{x}$ and $\boldsymbol{\alpha}$ spaces is not bijective. The convex hull of a nonlinear disjunction can, therefore, typically not be obtained through a $P$-split formulation. However, as illustrated by example \eqref{eq:example1}  $P$-split formulations can form good approximations of the convex hull. 
\begin{remark}
Theorem 5 is merely of theoretical interest. To construct the convex hull of a linear disjunction, the extended convex hull formulation \cite{balas1998disjunctive} should be directly applied. In this case, there is no advantage of the $P$-split formulation.
\end{remark}

\begin{corollary}
For disjunctions where the constraint functions have additive bounds and $\mathcal{X}$ is a box defined by bounds on all variables, the continuous relaxations of P-split formulations form a hierarchy. The relaxations start from the big-M relaxation ($P=1$) with each additional split resulting in an as tight or tighter relaxation, and with only linear constraints it converges to the convex hull relaxation ($P=n$). 
\end{corollary}
\begin{proof}
By Theorem~4, $1$-split and big-M have the same continuous relaxation in the $\boldsymbol{x}$-space. Each additional split results in an as tight or tighter relaxation. The equivalence to the convex hull is given by Theorem~5.
\end{proof}

The natural follow-up question to Theorem~5 is whether a $P$-split formulation can form the convex hull of a disjunction with convex constraints if $\mathcal{X}$ is a box. Theorem~6 shows that, in general, that is not the case. To simplify the theorem, we first introduce a simple definition of a disjunctive variable property.  

\begin{definition}
We say that the disjunctive constraint \eqref{eq:main_disjunction} is \textit{fully disjoint} if the intersection of any two disjuncts of the constraint is empty.
\end{definition}
\begin{theorem}
For a disjunctive constraint that is fully disjoint, with additive bounds and constraint functions that are strictly convex, a $P$-split formulation cannot form the true convex hull of the disjunctive constraint.
\end{theorem}
\begin{proof}
It is sufficient only to consider the $n$-split formulation, as the additive-bound property ensures the $n$-split has the strongest relaxation by Theorem 4. First, we pick two points $\boldsymbol{x}^1$ and $\boldsymbol{x}^2$ from two different disjuncts such that they are both extreme points of the disjuncts and that the entire line segment connecting $\boldsymbol{x}^1$ and $\boldsymbol{x}^2$ lies on the boundary of the convex hull of the disjunctive constraint. We will refer to the disjuncts as \say{1} and \say{2}.  We map these points onto the $P$-split representation of the disjunctive constraint (with $P=n$) by selecting
\begin{equation*}
\begin{aligned}
    &\alpha^{1,k}_i = h_{i,k}(x_i^1)   &&\forall\ i \in \{1,\dots, n\},  \ \forall k \in\mathcal{C}_1, \mathcal{C}_2\\
    &\alpha^{2,k}_i = h_{i,k}(x_i^2)   &&\forall\ i \in \{1,\dots, n\},  \ \forall k \in\mathcal{C}_1, \mathcal{C}_2,\\
\end{aligned}
\end{equation*}
and we aggregate these variables into the vectors $\boldsymbol{\alpha}^1$ and  $\boldsymbol{\alpha}^2$. By construction the points $(\boldsymbol{x}^1,\boldsymbol{\alpha}^1)$ and $(\boldsymbol{x}^2,\boldsymbol{\alpha}^2)$ are extreme points of the $n$-split representation of the disjunctive constraint. Now, consider the point in the middle of the line segment connecting the two points given by $(\boldsymbol{x}^3,\boldsymbol{\alpha}^3) = \left(0.5(\boldsymbol{x}^1 +\boldsymbol{x}^2), 0.5(\boldsymbol{\alpha}^1 + \boldsymbol{\alpha}^2)\right).$ As the functions $h_{i,k}$ are strictly convex, we get
\begin{equation*}
\begin{aligned}
   h_{i,k}(x_i^3) <  \alpha^{3,k}_i     &&\forall\ i \in \{1,\dots, n\},  \ \forall k \in\mathcal{C}_1, \mathcal{C}_2.
\end{aligned}
\end{equation*}
which follows directly from Jensen's inequality. Thus, $\exists\delta>0$ and a ball $\mathcal{B}_\delta =\{\boldsymbol{x} \in \R^n: \norm{\boldsymbol{x} - \boldsymbol{x}^3} \leq \delta \}$ such that 
\begin{equation*}
\begin{aligned}
   h_{i,k}(x_i) \leq  \alpha^{3,k}_i     &&\forall x_i \in \mathcal{B}_\delta, \  \forall\ i \in \{1,\dots, n\},  \ \forall k \in\mathcal{C}_1, \mathcal{C}_2.
\end{aligned}
\end{equation*}
Note that the point $\boldsymbol{\alpha}^3$ is within the feasible set of the $P$-split formulation, as it is a convex combination of two points that are both feasible for one disjunct each in the $P$-split representation. By construction, $\boldsymbol{x}^3$ lies on a facet of the convex hull of the disjunctive constraint. However, there is a small neighborhood around $\boldsymbol{x}^3$ defined by $\mathcal{B}_\delta$ that will be feasible for the $P$-split formulation. 
\end{proof}

From Theorem~6, we cannot expect a $P$-split formulation to exactly form the convex hull of disjunctive constraints with convex nonlinear functions, and this is not the goal of the $P$-split framework. Instead, the $P$-split framework offers a set of alternative formulations with a stronger continuous relaxation than big-M and computationally simpler formulations than the convex hull formulation. 

\subsection{Strengthening the relaxation by linking constraints}
Unless at least one $\alpha$ variable can be used in multiple constraints in the disjunct, there are no direct connections between the constraints within the disjuncts. 
When relations among $\alpha$ variables do exist, omitting connections between $\alpha$-variables of different constraints can result in a formulation with a weaker continuous relaxation (note the formulations are still correct for any integer feasible solution). 
Similar observations are made in the reformulation-linearization technique literature \cite{misener2014framework,sherali2013reformulation}.
With fewer splits, \ie the case desirable for practical implementations, it becomes more unlikely that the same $\alpha$ variable can be used in multiple constraints. 
We ask: how can $P$-split formulations be strengthened by introducing direct connections between constraints of the disjuncts?

For disjunctions with multiple constraints per disjunct, it can be possible to strengthen the $P$-split formulations with \textit{linking constraints}.
Linking constraints are derived from the upper and lower bounds obtained by combining two (or more) of the $\alpha$-variables from different constraints within the same disjunct  
\begin{equation}
\label{eq:linking}
    LB_\mathcal{X}(\rho_1 \alpha_s^{l,k} + \rho_2 \alpha_s^{l,j}) \leq \rho_1 \alpha_s^{l,k} + \rho_2 \alpha_s^{l,j} \leq UB_\mathcal{X}(\rho_1 \alpha_s^{l,k} + \rho_2 \alpha_s^{l,j}), 
\end{equation}
with $\rho_1, \rho_2 \in \{-1, 1\}$, $k,j \in \mathcal{C}_l$, and $LB _\mathcal{X}()$ and $UB_\mathcal{X}()$) are, respectively, the lower and upper bounds of the expressions over $\mathcal{X}$. 
Note the similarity to Belotti (2013) \cite{belotti2013bound}.
To limit the number of possible linking constraints, we only consider weighting factors $\rho_1, \rho_2 = \pm 1$. Linking constraints \eqref{eq:linking} can be included in the disjuncts, and they appear in the $P$-split formulation as 
\begin{equation}
    \begin{aligned}
    & \rho_1 \nu_d^{\alpha_s^{l,k}} + \rho_2 \nu_d^{\alpha_s^{l,j}} \leq UB_\mathcal{X}(\rho_1 \alpha_s^{l,k} + \rho_2 \alpha_s^{l,j})\lambda_d,\\
    & \rho_1 \nu_d^{\alpha_s^{l,k}} + \rho_2 \nu_d^{\alpha_s^{l,j}} \geq LB_\mathcal{X}(\rho_1 \alpha_s^{l,k} + \rho_2 \alpha_s^{l,j})\lambda_d.
    \end{aligned}
\end{equation}
The number of possible linking constraints grows rapidly with the number of constraints per disjunct. Therefore, it is important to identify which combinations of $\alpha$ variables produce redundant linking constraints and which ones can be useful. Proposition \ref{prop:linking} provides a criterion for identifying redundant combinations.
\begin{proposition}\label{prop:linking}
Any combination of auxiliary variables $\alpha_s^{l,k}$ and $\alpha_s^{l,j}$ such that
 \begin{equation}
\label{eq:redundant_test}
      \begin{aligned}
     &LB_\mathcal{X}(\rho_1 \alpha_s^{l,k} + \rho_2 \alpha_s^{l,j}) =  LB_\mathcal{X}(\rho_1 \alpha_s^{l,k}) +  LB_\mathcal{X}( \rho_2 \alpha_s^{l,j}),\\
     &UB_\mathcal{X}(\rho_1 \alpha_s^{l,k} + \rho_2 \alpha_s^{l,j}) =  UB_\mathcal{X}(\rho_1 \alpha_s^{l,k}) +  UB_\mathcal{X}( \rho_2 \alpha_s^{l,j}),
 \end{aligned}
 \end{equation}
results in a redundant linking constraint.
\end{proposition}
\begin{proof}
For any combination such that \eqref{eq:redundant_test} hold, the resulting linking constraint can be obtained by aggregating the corresponding bound constraints. 
\end{proof}
Proposition \ref{prop:linking} suggests two rules for selecting combinations of linking constraints:
\begin{enumerate}
    \item Only combinations such that the corresponding partitioning of the $\boldsymbol{x}$ variables overlap can result in a non-redundant linking constraint.
    \item For the simple case of linear disjunctions over a box domain $\mathcal{X}$, only combinations such that at least one pair of the $\boldsymbol{x}$-variables in the expressions representing the $\alpha$-variables have different signs can result in non-redundant linking constraints.
\end{enumerate}

\begin{figure}[t]
    \centering
    \begin{subfigure}[t]{.32\linewidth}
    \includegraphics[trim={1.5cm 0cm 2cm 0cm},clip,width=.99\linewidth]{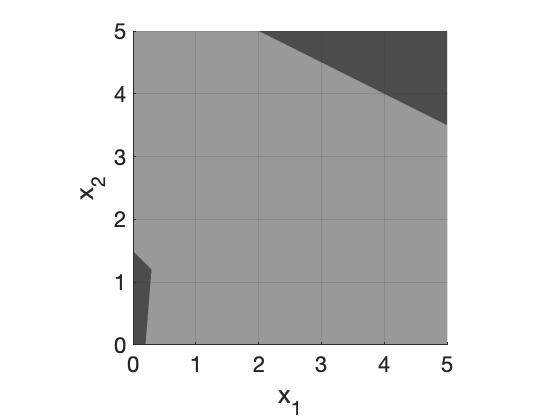}
    \caption{\centering 1-split/big-M}
    \end{subfigure}
    \begin{subfigure}[t]{.32\linewidth}
    \includegraphics[trim={1.5cm 0cm 2cm 0cm},clip,width=.99\linewidth]{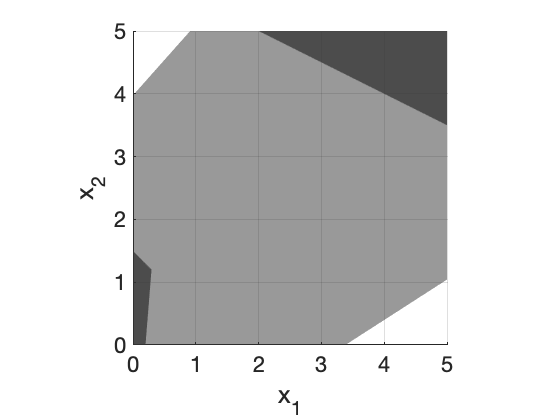}
    \caption{\centering 2-split}
    \end{subfigure}
    \begin{subfigure}[t]{.32\linewidth}
    \includegraphics[trim={1.5cm 0cm 2cm 0cm},clip,width=.99\linewidth]{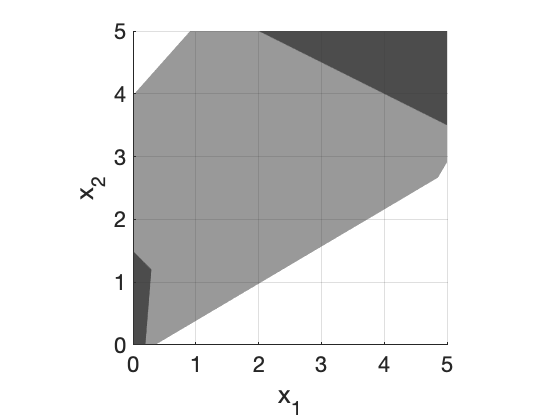}
    \caption{\centering 2-split + linking constraints}
    \end{subfigure}
      \caption{The dark regions show the feasible set of \eqref{eq:example2} in the $(x_1,x_2)$-space. The light grey areas show the continuously relaxed feasible set of the $1$-split/big-M formulation, $2$-split formulation, and $2$-split formulation with linking constraints.}
    \label{fig:relaxations_linkin}
\end{figure}
\subsubsection{Illustrative example}
A simple example illustrates how linking constraints can affect the continuous relaxation of a $P$-split formulation. Consider the disjunctive constraint
\begin{equation}
\label{eq:example2}
\tag{ex-2}
\begin{aligned}
    &\begin{bmatrix}
     x_1 + x_2 + x_3 + x_4  \leq 1.5 \\
     1.5x_1 - 1.2x_2 + x_3 - x_4  \leq -1
     \\
    \boldsymbol{x} \in [0
    ,\ 5]^4
    \end{bmatrix}
    \
    \lor \
    \begin{bmatrix}
    -1x_1 - 2x_2 - x_3 - 2x_4 \leq -26\\
    -2x_1 +x_2 +x_3 -0.5x_4 \leq -1\\
    \boldsymbol{x} \in [0
    ,\ 5]^4
    \end{bmatrix}.
\end{aligned}
\end{equation}
We form a $2$-split formulation, using the variable partitioning $\{x_1, x_2\}$ and $\{x_3, x_4\}$, and the tightest valid big-M formulation with individual M-coefficients for each constraint. To this $2$-split formulation, we also add the above linking constraints. We only consider linking constraints between the $\alpha$-variables within each disjunct, which results in 16 linking constraints. 
Figure~2 depicts the continuous relaxation of the $2$-split formulation with and without the linking constraints: the linking constraints result in a significantly tighter relaxation.

The simple example \eqref{eq:example2} shows that linking constraints can significantly strengthen a $P$-split formulation for a disjunction with multiple constraints per disjunct. Even with the selection rules, the number of non-redundant linking constraints can increase significantly with both $P$ and the number of constraints per disjunct. Fortunately, the number of potential linking constraints is independent of the number of variables in the original disjunction. For a moderate number of splits and constraints per disjunct, it can be favorable to include all linking constraints that satisfy the two non-redundant criteria suggested by Proposition 2. Otherwise, a possible heuristic selection rule could be to greedily select the linking constraints with the largest deviations between the combined and separate upper/lower bounds of the $\alpha$-variables. More advanced strategies for selecting linking constraints are beyond the scope of this paper.

\section{Beyond convex disjuncts}\label{sec:5}
So far we have assumed that all constraints $h_{i,k}$ are convex. But, the $P$-split framework together with convex underestimators can directly be applied to disjunctive constraints with nonconvexities within the disjuncts to form a valid relaxation. We prove that the main properties of the $P$-split formulations still hold with nonconvex constraint functions if we limit the variables to a bounded box $\mathcal{B}$, defined by the variables' upper and lower bounds, instead of a more general convex set $\mathcal{X}$.

Given a disjunctive constraint where some of the functions $h_{i,k}$ are nonconvex, a natural approach is to first convexify the constraints of each disjunct to obtain 
\begin{equation}
\begin{aligned}
\label{eq:n-conv_disjunction}
    &\underset{l \in \mathcal{D}}{\lor} \begin{bmatrix} \underset{\mathcal{B}}{\text{convenv}} \left(\sum\limits_{i=1}^{n} h_{i,k}(x_i)\right) \leq b_{l,k} \quad \forall k \in \mathcal{C}_{l}\\ 
    \boldsymbol{x} \in \mathcal{B}
    \end{bmatrix}\\
\end{aligned},
\end{equation}
where $\underset{\mathcal{B}}{\text{convenv}}(f)$ denotes the convex envelope of function $f$ over the box $\mathcal{B}$.  First convexifying the constraints of each disjunct and then convexifying the disjunctive constraint makes sense, especially from a computational point of view. An important justification for this approach is stated in the following proposition, where  $\text{convh}()$ denotes the convex hull operator.
\begin{proposition} \label{prop:prop3}
    If $\mathcal{B}\bigcap
    \left(\bigcap\limits_{k \in \mathcal{C}_{l}}
      \left\{ \boldsymbol{x} \in \R^n :\ 
    \underset{\mathcal{B}}
{\mathrm{convenv}}\left(\sum\limits_{i=1}^{n} h_{ik}(x_i)\right) \leq b_{l,k} \right\}\right) =\\  {\mathrm{convh}}\left(\mathcal{B} \bigcap\limits_{k \in \mathcal{C}_{l}} \left\{ \boldsymbol{x} \in \R^n :\
 \sum\limits_{i=1}^{n} h_{ik}(x_i) \leq b_{l,k} \right\}\right)$ for all $l\in \mathcal{D}$, then 
    
    \begin{multline*}
            \mathrm{convh} \left(\underset{l \in \mathcal{D}}{\lor} \begin{bmatrix} \underset{\mathcal{B}}{\mathrm{convenv}} \left(\sum\limits_{i=1}^{n} h_{ik}(x_i)\right) \leq b_{l,k} \ \forall k \in \mathcal{C}_{l}\\ 
    \boldsymbol{x} \in \mathcal{B}
    \end{bmatrix}\right) =   \\   \mathrm{convh} \left(\underset{l \in \mathcal{D}}{\lor} \begin{bmatrix}  \sum\limits_{i=1}^{n} h_{ik}(x_i) \leq b_{l,k} \ \forall k \in \mathcal{C}_{l}\\ 
    \boldsymbol{x} \in \mathcal{B}
    \end{bmatrix}\right).
    \end{multline*}
\end{proposition}
\begin{proof}
    Follows trivially from the fact that the convex hull operator on a set does not introduce any new extreme points. 
\end{proof}

However, taking the convex hull of the convexified disjunctive constraint \eqref{eq:n-conv_disjunction} is not guaranteed to give the convex hull of the original disjunctive constraint, because, in general $\big{\{} {\text{convh}}\left(g(x) \leq 0 \cap  \mathcal{B}\right)\big{\}} \subset \big{\{}\underset{\mathcal{B}}{\text{convenv}}\left(g(x)\right) \leq 0 \big{\}}$. Therefore, the sequential convexification can result in a further relaxation that needs to be handled by another technique, such as spatial branching, to guarantee a correct solution. 
Note that Propositon~\ref{prop:prop3} holds not only for a bounded box $\mathcal{B}$, but for convex sets. However, we state the proposition for set $\mathcal{B}$ for clarity and to match the other results. 

To construct and analyze the resulting $P$-split formulation we need a property of the convex envelope which is formally stated in the following proposition. 

\begin{proposition} (Theorem IV.8, Horst and Tuy \cite{horst1996global})
\label{prop:prop4}
The convex envelope of a separable function over a box domain can be decomposed into a sum of convex envelopes, 
$$\underset{\mathcal{B}}{\mathrm{convenv}} \left(\sum\limits_{i=1}^{n} h_{ik}(x_i)\right) = \sum\limits_{i=1}^{n} \underset{\mathcal{B}}{\mathrm{convenv}}\left(h_{ik}(x_i)\right).$$    
\end{proposition}
An important consequence of Proposition~\ref{prop:prop4} is that we can take the convex envelopes from any partitioning of the summed components and obtain an equally strong convex underestimator overall. Thus, we can consider the constraints of each disjunct in \eqref{eq:n-conv_disjunction} to be convex and fully separable and deal with them as in the previous section. 

To form a $P$-split formulation we take the convex envelope of the partial sum defined by the split partitionings $\mathcal{I}_s$, and the constraints in the disjuncts will then be partitioned according to 
\begin{align}
    &\underset{\mathcal{B}}{\text{convenv}}\left(\underset{i \in \mathcal{I}_s}{\sum}h_{i,k}(x_i)\right) \leq \alpha^{l,k}_s,  \label{eq:n-con_con1}\\
    &\sum_{s=1}^P \alpha^{l,k}_s  \leq b_{l,k}.
\end{align}
The only difference in the resulting $P$-split formulation is that we replace the constraints on line 4 in \eqref{eq:p-split-multicon} with \eqref{eq:n-con_con1}. Corollary \ref{cor:4} shows that the relaxation strength properties for the $P$-split formulations also hold when applied to the disjunctive constraint \eqref{eq:n-conv_disjunction}.

\begin{corollary}\label{cor:4}
A 1-split formulation is equivalent to the big-M formulation when applied to \eqref{eq:n-conv_disjunction}. Furthermore, splitting one group of variables in a $P$-split formulation, to form a
$(P+1)$-split formulation, results in an as tight or tighter continuous relaxation if the
bounds on the $\alpha$-variables are additive.
\end{corollary}
\begin{proof}
From Proposition~\ref{prop:prop4}, it is clear that the convex envelopes can be formed separately on any partitioning of the constraints without affecting the relaxation strength. The rest of the proof follows trivially from Theorem~\ref{thm:p_split_mult}.  
\end{proof}
Under specific circumstances, it is possible to obtain the convex hull of the disjunctive constraint~\eqref{eq:n-conv_disjunction} through a $P$-split formulation. This is summarized in the following corollary.
\begin{corollary}
If $\underset{\mathcal{B}}{\mathrm{convenv}} \left(\sum\limits_{i=1}^{n} h_{ik}(x_i)\right)$ is an 
affine function, or representable by linear constraints, and
\begin{multline*}
\left\{ \mathcal{B} \bigcap\limits_{k \in \mathcal{C}_{l}} \left\{ \boldsymbol{x} \in \mathbb{R}^n : \underset{\mathcal{B}}{\mathrm{convenv}}\left(\sum\limits_{i=1}^{n} h_{ik}(x_i)\right) \leq b_{l,k}\right\} \right\} \\= \left\{ {\mathrm{convh}}\left(\mathcal{B} \bigcap\limits_{k \in \mathcal{C}_{l}}\left\{ \boldsymbol{x} \in \mathbb{R}^n : 
 \sum\limits_{i=1}^{n} h_{ik}(x_i) \leq b_{l,k}\right\} \right)\right\},
 \end{multline*}
 then a continuous relaxation of the $n$-split formulation forms the convex hull of \eqref{eq:n-conv_disjunction}. 
\end{corollary}
\begin{proof}
As the intersection of the resulting linear constraints in each disjunct\\ $\underset{\mathcal{B}}{\text{convenv}} \left(\sum\limits_{i=1}^{n} h_{ik}(x_i)\right) \leq b_{l,k}$ forms the convex hull of the feasible set of each disjunct, the proof follows directly from Theorem~\ref{thm:sufficient}.
\end{proof}

Finally, a practical comment about the case with non-convex constraints in the disjuncts. If a $P$-split formulation is used in conjunction with a global MINLP solver based on convex/concave underestimators, then the \say{convex} formulation \eqref{eq:p-split-multicon} can be directly given to the solver, as the solver will handle the convex relaxation in \eqref{eq:n-con_con1} on its own, and all the properties mentioned above hold. 

\section{Computational considerations and further enhancements} \label{sec:6}
The previous sections describe the general framework for generating $P$-split formulations. This section focuses on more specific implementation details, such as how to partition the original variables into the partitioning sets $\mathcal{I}_1, \dots, \mathcal{I}_P$. These details may have a significant impact on the computational performance.

If the problem is close to symmetric with respect to all variables of a disjunctive constraint, \ie similar range and coefficients/expressions for all variables, then there is no reason to believe that one partitioning should have a computational advantage over another. But, for some problems, certain variables can have a much stronger impact on the constraints of the disjuncts. For example, weights on the inputs to ReLU neural networks may differ by orders of magnitude. Our conference paper focusing on neural networks \cite{tsay2021partition} investigates several partitioning strategies. The strategy of partitioning the variables into sets of equal size and grouping variables with similar coefficients into the same partitioning was computationally favorable, and was also motivated based on a theoretical analysis of the resulting continuous relaxation \cite{tsay2021partition}. This partitioning strategy is also available in our open-source code OMLT \cite{ceccon2022omlt} for neural networks.
Variables can also be partitioned into sets of unequal size, but without prior information that a small subset of variables are more important, it is unlikely that such a strategy would be favorable.

Bounds on the auxiliary variables are important in the $P$-split formulation, but for simplicity and general applicability we have only considered globally valid bounds. The bounds $ \ubar{\alpha}^l_s$ and $ \bar{\alpha}^l_s$ in \eqref{eq:bounds_alpha} are simply defined as upper and lower bounds of the individual expressions over the convex set $\mathcal{X}$. As such, the bounds $\ubar{\alpha}^l_s$ and $ \bar{\alpha}^l_s$ do not account for dependencies among the auxiliary variables introduced by the constraints of the disjuncts. It would be possible to utilize different bounds for the $\alpha$-variables in each disjunct. Such bounds could potentially be significantly tighter, as they must only be valid within the specific disjuncts; however, practically obtaining such bounds for the auxiliary variables becomes a non-trivial task. These tighter bounds can be obtained through optimization-based bounds tightening (OBBT), but unless the disjunct is linear, some of the OBBT problems will be nonconvex. For disjunctions with only linear constraints, and for small $P$, it may be favorable to perform OBBT on the auxiliary variables to capture dependencies and obtain stronger local bounds. Note that locally valid bounds on the auxiliary variables can be directly included in the $P$-split formulation, without any extra constraints, and should always be used if readily available. 

\begin{figure}[t]
    \centering
    \begin{subfigure}[t]{.32\linewidth}
    \includegraphics[trim={1.5cm 0cm 2cm 0cm},clip,width=.99\linewidth]{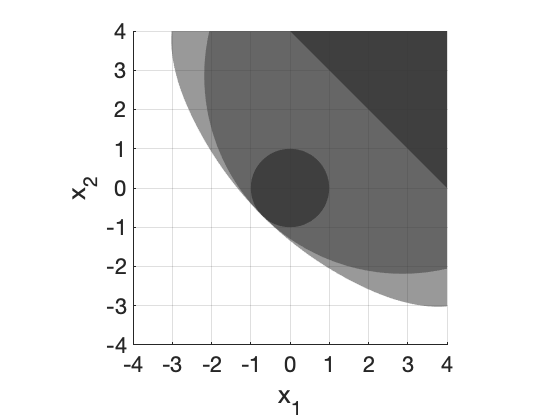}
    \caption{\centering $2$-split $\left(\{x_1, x_2\}, \{x_3, x_4\}\right)$}
    \end{subfigure}
    \begin{subfigure}[t]{.32\linewidth}
    \includegraphics[trim={1.5cm 0cm 2cm 0cm},clip,width=.99\linewidth]{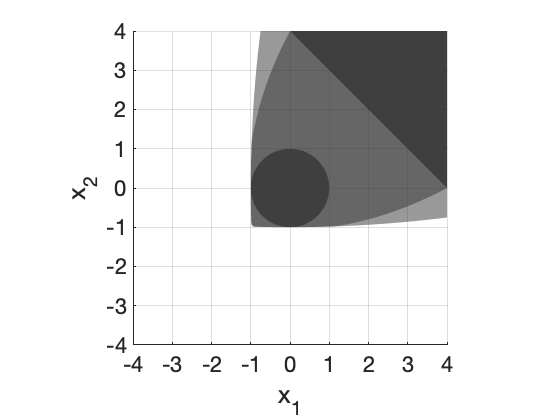}
    \caption{ $4$-split}
    \end{subfigure}
      \caption{The figures show the feasible sets of the continuous relaxations of the $2$-split and $4$-split formulations with and without the local bounds for \eqref{eq:example1}. The darker regions show the relaxations of the $P$-split formulations with the tighter local bounds. For the $1$-split, it is not possible to obtain tighter local bounds. }
    \label{fig:relaxations_independent}
\end{figure}
Our numerical comparisons do not include locally valid bounds on the $\alpha$-variables, as such bounds are typically not easily obtainable and could give the $P$-split formulations an unfair advantage. But to illustrate the potential, we have determined locally valid bounds for all the $\alpha$-variables for  \eqref{eq:example1} and used them in the $P$-split formulation. Figure~3 shows how the tighter locally valid bounds strengthen the continuous relaxation notably for both the $2$-split and $4$-split formulations. Due to the nonlinear constraints in the disjunct, we cannot recover the true convex hull of the disjunction in \eqref{eq:example1} through a $P$-split formulation. But the $4$-split with locally valid bounds is a good approximation of the convex hull.

\section{Numerical comparison}\label{sec:7}

The previous sections compare the $P$-split formulations against the big-M and convex hull in terms of relaxation strength. But the computational performance of a formulation also depends on complexity of its continuous relaxation. To test and compare the formulations, we apply the $P$-split, big-M, and convex hull formulations to a set of 344 test problems. We consider three challenging classes of problems matching Assumptions 1--4. 
The $P$-split framework applies to general disjunctions with nonlinear constraints, but we limit the numerical comparison to disjuncts with linear or convex quadratic constraints \ie problems where the convex hull of the disjunction is representable by a polyhedron or (rotated) second-order cone constraints \cite{ben2001lectures}. This ensures that the convex hull formulation can be directly solved by software such as Gurobi for comparison. This highlights a computational advantage of the $P$-split formulations: they avoid some numerical difficulties typically associated with the convex-hull formulation of nonlinear disjunctions \cite{sawaya2007computational} by linearizing the constraints of the disjunction. 

\subsection{Test instances}
We first briefly describe the problems used in the numerical tests. 
These specific problem classes allow us to construct non-trivial, but feasible, optimization problems of varying dimensionality and difficulty.   

\subsubsection{P\_ball problems}
The P\_ball problems are challenging disjunctive test problems \cite{kronqvist2020disjunctive} which assign $P$-points to $n$-dimensional unit balls such that the total $\ell_1$ distance between all points is minimized and only one point is assigned to each unit ball. The combinatorial difficulty depends on the numbers of balls and points, and the numerical difficulties increase with the dimensionality of the balls. Higher dimensonality leads to more nonlinear terms, and the upper bounds on the $\alpha$-variables, including the big-M coefficients, increase with the number of variables.
The disjunctive structure 
arises from assigning each point to a ball, and each disjunct contains a convex quadratic constraint.

Given the number of $d$-dimensional balls $n_b$ with centers $\boldsymbol{c}^1, \dots, \boldsymbol{c}^{n_b}$ and number of points $n_p$, the P\_ball problems in \cite{kronqvist2020disjunctive} can be written as 
\begin{equation}
    \label{eq:P-ball}
    \begin{aligned}
        & \min &&\sum\limits_{i=1}^{n_b}\sum\limits_{j=i+1}^{n_b} \norm{\boldsymbol{p}^i - \boldsymbol{p}^j}_1\\
        &\text{s.t.} && \underset{i\in\mathcal{N}_b}{\lor}\begin{bmatrix}
            \norm{\boldsymbol{c}^i - \boldsymbol{p}^j}_2^2 \leq 1\\
            b_{ij} = 1
        \end{bmatrix} && j = 1, \dots, n_p,\\
        & && \sum\limits_{j=1}^{n_p} b_{ij} = 1 && i = 1, \dots n_b,\\
        & && \boldsymbol{p}^j \in [0, 10]^d && j = 1, \dots, n_p, \\
        & && b_{ij}\in \{0,1\} &&  i= 1, \dots n_b, j = 1, \dots, n_p,
    \end{aligned}
\end{equation}
where $\boldsymbol{p}^j$ are the points to be assigned to different balls, $\mathcal{N}_b = \{1,\dots,n_b\}$ and the variables $b_{ij}$ are used to ensure each ball is only assigned one point. More details and code for generating P\_ball problems area available on \url{github.com/jkronqvi/points_in_circles}.

By introducing auxiliary variables for the differences between the points and the centers, we express the convex hull by second-order cone constraints \cite{ben2001lectures} in a form suitable for Gurobi. Bounds on the auxiliary variables in the P-split formulation are obtained using the same technique for deriving big-M coefficients in \cite{kronqvist2020disjunctive}, but in the subspace corresponding to the auxiliary $\alpha$-variable. For these problems, it is possible to obtain tight bounds on the auxiliary variables (and tight big-M coefficients) by simply analyzing closed form expressions. Table~\ref{tab:test_probs} gives specific details on the P\_ball problems used in the numerical comparison.

\subsubsection{Finding optimal sparse input features of ReLU-NNs}
Mixed-integer programming has been used for optimizing over trained neural networks (NNs) to quantify extreme outputs, compress networks, and solve verification problems \cite{anderson2020strong,botoeva2020efficient,serra2020lossless}. 
We refer the reader to Huchette et al. \cite{huchette2023deep} for an overview of applications and methods. 
These optimization problems are difficult due to a potentially weak continuous relaxation \cite{anderson2020strong}, and intermediate formulations between big-M and convex hull show promising results \cite{tsay2021partition}. 
This paper uses a similar framework to identify \textit{optimal sparse input features} (OSIF)~\cite{zhao2023model}. The goal is to determine a sparse input that maximizes the prediction of a certain output. We consider NNs trained on the MNIST data set \cite{lecun2010mnist}. 
The goal is then to maximize the probability of the input being classified as a specific number by coloring a limited number of input-image pixels, with the number of pixels limited by the $\ell_1$-norm. Figure~\ref{fig:relaxations_independent_digit} illustrates optimal sparse input features for a NN with varying $\ell_1$-norm. 

The disjunction in the optimization problems arises from encoding the switching behavior of the ReLU nodes. For a NN with $L$ hidden layers the OSIF problem can be written as 
\begin{equation}
\begin{aligned}
     \max_{\boldsymbol{x}_l} &\left(\boldsymbol{w}_L^j\right)^\top \boldsymbol{x}_L\\
    \text{s.t.} &\begin{bmatrix}
    x_{l+1}^i = \left(\boldsymbol{w}_l^i\right)^\top\boldsymbol{x}_l + b_l^i \\ x_{l+1}^i \geq 0
    \end{bmatrix}
    \lor 
    \begin{bmatrix}
\left(\boldsymbol{w}_l^i\right)^\top\boldsymbol{x}_l + b_l^i\leq 0\\
x_{l+1}^i = 0
    \end{bmatrix} && i \in \mathcal{N}_l;\ l 
    \in [0,...,L-1] \\
    & \norm{\boldsymbol{x}_0}_1 \leq \kappa, \quad \boldsymbol{x}_0\in[0, 1]^n, \ \boldsymbol{x}_l\in \mathbb{R}_+^{|\mathcal{N}_l|}\ \ \forall l \in [1,...,L]
    \end{aligned}
\end{equation}
where $\boldsymbol{x}_0$ is the input vector and $\boldsymbol{x}_l$ denotes the outputs of layer $l$. Index $j$ is the classification category that we are maximizing the probability for, and $\kappa$ is a parameter limiting the coloring of input pixels.  Sets $\mathcal{N}_l$ contain the indexes of the nodes in the $l$-th layer, and  $\boldsymbol{w}_l^i$, $ b_l^i$ are the trained weights and biases. For more details of mixed-integer ReLU-NN encodings, see \cite{anderson2020strong,tsay2021partition}.

\begin{figure}[t]
    \centering
    \begin{subfigure}[t]{.16\linewidth}
    \includegraphics[trim={3.5cm 1.1cm 2cm 0cm},clip,width=.99\linewidth]{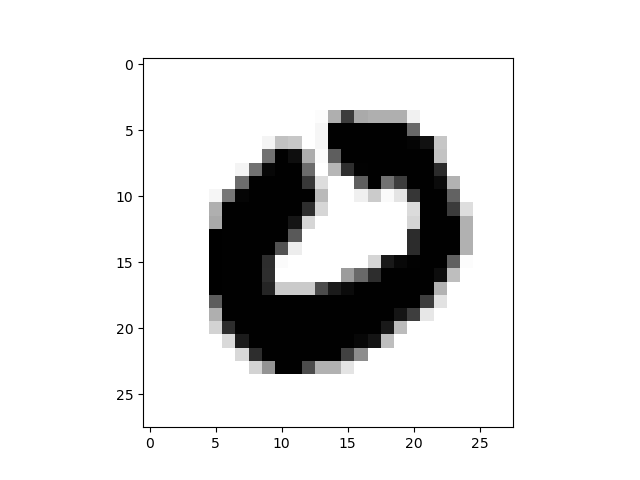}
    \caption{\centering MNIST sample image}
    \end{subfigure}
    \begin{subfigure}[t]{.16\linewidth}
    \includegraphics[trim={3.5cm 1.1cm 2cm 0cm},clip,width=.99\linewidth]{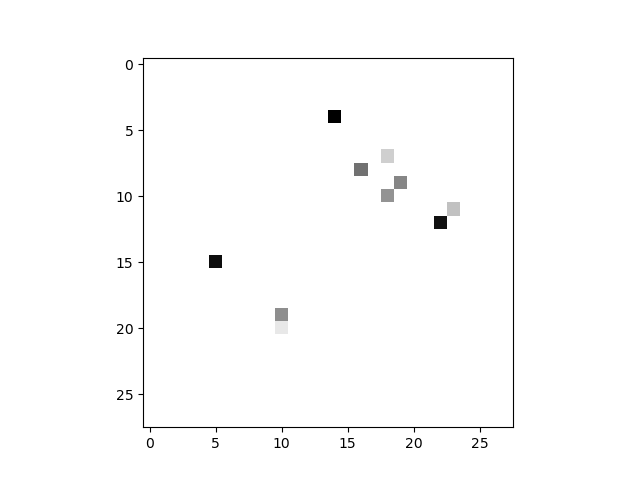}
    \caption{\centering $\norm{\cdot}_1 \leq 5$}
    \end{subfigure}
    \begin{subfigure}[t]{.16\linewidth}
    \includegraphics[trim={3.5cm 1.1cm 2cm 0cm},clip,width=.99\linewidth]{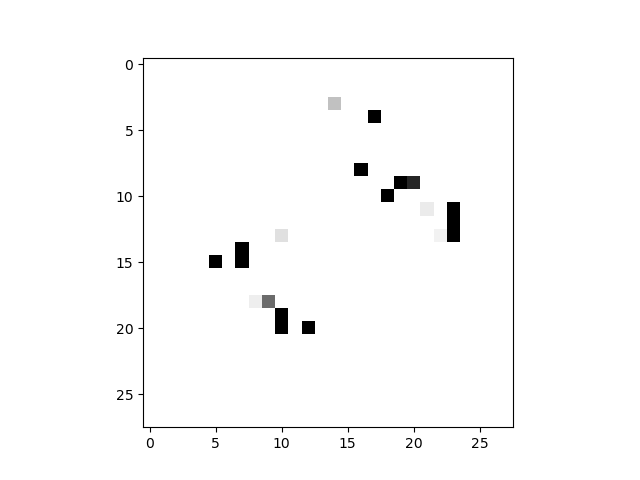}
    \caption{\centering  $\norm{\cdot}_1 \leq 15$}
    \end{subfigure}
       \begin{subfigure}[t]{.16\linewidth}
    \includegraphics[trim={3.5cm 1.1cm 2cm 0cm},clip,width=.99\linewidth]{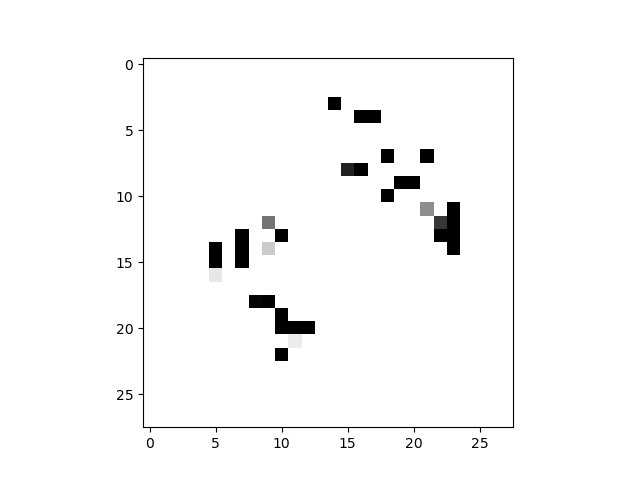}
    \caption{\centering  $\norm{\cdot}_1 \leq 30$}
    \end{subfigure}

    \begin{subfigure}[t]{.16\linewidth}
    \includegraphics[trim={3.5cm 1.1cm 2cm 0cm},clip,width=.99\linewidth]{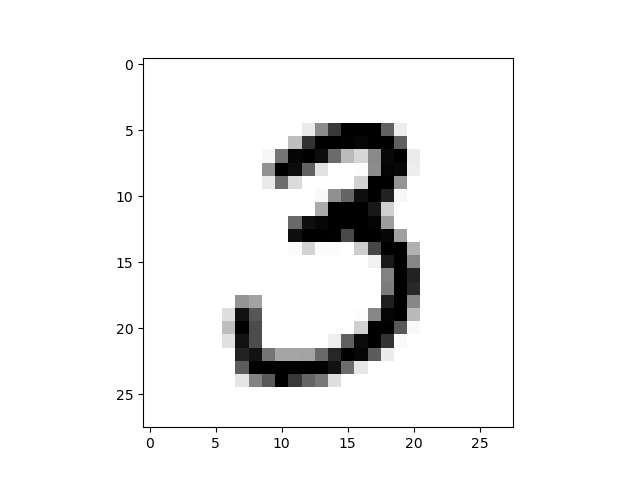}
    \caption{\centering MNIST \phantom{a} sample image}
    \end{subfigure}
    \begin{subfigure}[t]{.16\linewidth}
    \includegraphics[trim={3.5cm 1.1cm 2cm 0cm},clip,width=.99\linewidth]{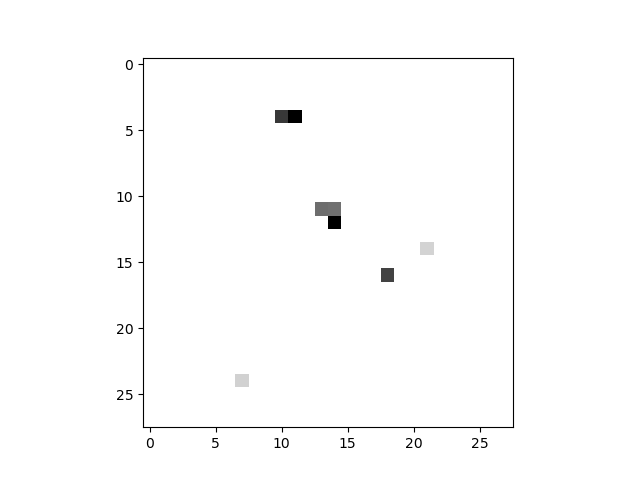}
    \caption{\centering $\norm{\cdot}_1 \leq 5$}
    \end{subfigure}
    \begin{subfigure}[t]{.16\linewidth}
    \includegraphics[trim={3.5cm 1.1cm 2cm 0cm},clip,width=.99\linewidth]{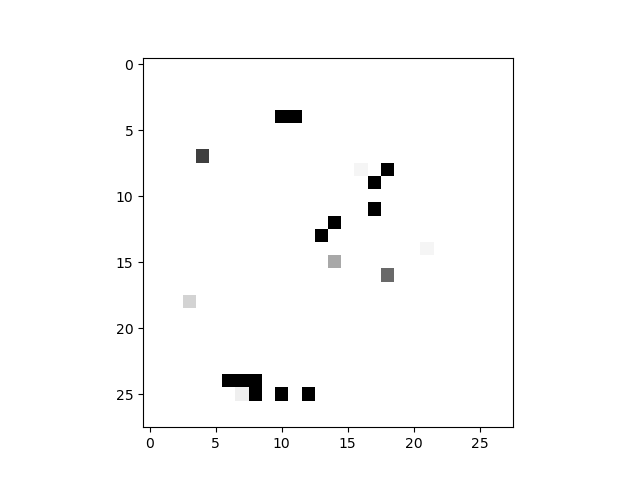}
    \caption{\centering  $\norm{\cdot}_1 \leq 15$}
    \end{subfigure}
       \begin{subfigure}[t]{.16\linewidth}
    \includegraphics[trim={3.5cm 1.1cm 2cm 0cm},clip,width=.99\linewidth]{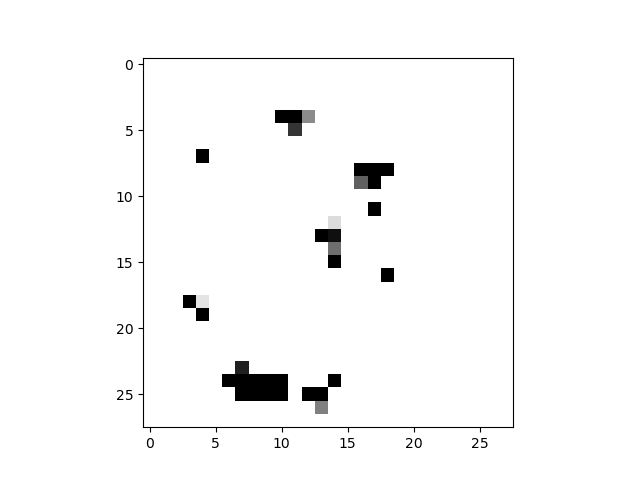}
    \caption{\centering  $\norm{\cdot}_1 \leq 30$}
    \end{subfigure}
      \caption{Optimal sparse input features for numbers 0 and 3 with different $\ell_1$-constraints on the input, along with two sample images. The NN has two hidden layers with 50 nodes in each, and is available: \url{github.com/cog-imperial/PartitionedFormulations_OSIF}.} 
    \label{fig:relaxations_independent_digit}
\end{figure}

Our numerical tests use NNs with two hidden layers, with either 50 or 75 nodes per layer, trained to classify handwritten digits from the MNIST data set. The networks were implemented and trained using the Adadelta solver in PyTorch \cite{pytorch}. Default hyperparameter values were used, except for the learning rate, which was set to 1. For each NN, we optimize the input for each classification category, resulting in ten optimization problems per network (classes for digits 0 through 9). For the numerical tests, we limit the $\ell_1$-norm of the input to be less than 5 or 10, which is mainly to keep the problems solvable within a reasonable time limit. For the NNs there are several alternatives for deriving bounds, e.g., bound propagation and interval arithmetic, $\alpha$-CROWN \cite{xu2020fast}, or optimization-based bounds tightening (OBBT) by solving LP subproblems \cite{tsay2021partition}. We test both bounds computed by interval arithmetic and stronger (but computationally expensive) bounds computed by OBBT to also compare the impact of the bound strength on the formulations. We partition the variables using the Tsay et al. \cite{tsay2021partition} strategy based on node weights. Table~\ref{tab:test_probs} gives details on the OSIF problems used in the numerical comparison. The networks used and the code for generating the test problems are available on \url{github.com/cog-imperial/PartitionedFormulations_OSIF}.

\subsubsection{K-means clustering}
K-means clustering \cite{macqueen1967some} is a classical problem in data analysis and machine learning that is known to be NP-hard \cite{aloise2009np}. The goal is to partition data points into clusters such that the variance within each cluster is minimized. Using the Papageorgiou and Trespalacios \cite{papageorgiou2018pseudo} convex disjunctive programming formulation
\begin{equation}
\label{eq:clustering}
    \begin{aligned}
        & \min_{\mathbf{r} \in \mathbb{R}^L, \boldsymbol{x}^j \in \mathbb{R}^n} &&\sum_{i=1}^L r_i\\
        & \text{s.t.} &&\underset{j \in \mathcal{K}}{\lor} \begin{bmatrix}
        \begin{aligned}
        &\norm{\boldsymbol{x}^j - \mathbf{d}^i}_2^2 \leq r_i\vspace{0.1cm}\\
        &r_i \leq \underset{l \in \{1, \dots L\}}{\max}\norm{\mathbf{d}^{l} - \mathbf{d}^i}_2^2 
        \end{aligned}
        \end{bmatrix}\quad  \forall i \in \{1, \dots, L\},\\
    \end{aligned}
\end{equation}
where $\boldsymbol{x}^j$ are the cluster centers, $\{\mathbf{d}^i\}_{i=1}^L$ are $n$-dimensional data points, and $\mathcal{K} =\{1,2,\dots k\}$. For generating test instances, we used the G2 data set \cite{G2sets,franti2018k} and the MNIST data set. The G2 data set contains clustered data of different dimensionality: we randomly selected a fixed number of data points to generate our test instances. For MNIST, each data point represents an image of a handwritten number with 784 pixels. For the MNIST-based problems, we select the first images of each class ranging from 0 to the number of clusters.

For both the big-M and $P$-split formulations, we used the tightest globally valid bounds. The lower bounds are trivial, and the tightest upper bounds for the auxiliary variables in the P-split formulations are given by the largest squared Euclidean distance between any two data points in the subspace corresponding to the auxiliary variable. For the convex hull formulation, we introduce auxiliary variables for the differences $(\boldsymbol{x}-\mathbf{d})$ to express the convex hull of the disjunctions by rotated second-order cone constraints \cite{ben2001lectures}. For these problems, the $n$-split formulation and the convex hull formulation have an equal number of variables, but different types of constraints. Table~\ref{tab:test_probs} gives specific details.

\begin{table}[h]
\centering
\begin{tabular}{ c| c | c |c } \hline
 name &\ data points \ & \ data dimension \ & \ number of clusters\\
 \hline
 Cluster\_g1 & 20 & 32 & 2\\ 
 Cluster\_g2 & 25 & 32 & 2\\ 
 Cluster\_g3 & 20 & 16 & 3\\ \hline
 Cluster\_m1 & 5 & 784 & 3\\ 
 Cluster\_m2 & 8 & 784 & 2\\ 
 Cluster\_m3 & 10 & 784 & 2\\
 \hline
  &\ number of balls \ & \ number of points \ & \ ball dimension\\
 \hline
 P\_ball\_1\ \ & 8 & 5 & 32\\ 
 P\_ball\_2\ \ & 8 & 4 & 48\\ 
\hline
  &\ input dimension ($d$) \ & \ hidden layers \ & \ $\ell_1$-constraint \\
  \hline
 NN1\_OSIF & 784 & [50, 50] & $\norm{\cdot}_1 \leq 5$ \\
 NN2\_OSIF & 784 & [50, 50] & $\norm{\cdot}_1 \leq 10$ \\
 NN3\_OSIF & 784 & [75, 75] & $\norm{\cdot}_1 \leq 5$ \\
 NN4\_OSIF & 784 & [75, 75] & $\norm{\cdot}_1 \leq 10$ \\
 \hline
\end{tabular}
\caption{Details of the clustering, P$\_$ball, and neural network problems. For the clustering and P\_ball problems we randomly generated 15 instances of each class. The clustering problems with a ``$\_$g'' in the name originate from the G2 data set, and the ones with an ``$\_$m'' from MNIST. For each NN architecture, we trained 5 different NNs, and we optimize each of the 10 classifications individually, giving 50 instances of each class. This forms a test set with 320 optimization problems.} 
\label{tab:test_probs}
\end{table}

\subsubsection{Semi-supervised K-means clustering}
To test the impact of multiple constraints per disjunct, we also consider a modified version of the K-means clustering problem. Instead of assigning individual points to clusters, we assign groups of points to clusters. These problems arise in, \eg semi-supervised clustering \cite{bair2013semi,basu2004probabilistic} where we have known subsets of points in which all points have the same label. See figure Figure~\ref{fig:relaxations_independent_semi} for an illustration of this semi-supervised clustering task. For simplicity, we assume that we have $L$ such subsets that each contains $m$ points. Similar to the K-means clustering problem \eqref{eq:clustering}, this semi-supervised clustering problem can be formulated as

\begin{equation}
\label{eq:G_clustering}
    \begin{aligned}
        & \min_{\mathbf{r} \in \mathbb{R}^L, \boldsymbol{x}^j \in \mathbb{R}^n} &&\sum_{i=1}^L r_i\\
        & \text{s.t.} &&\underset{j \in \mathcal{K}}{\lor} \begin{bmatrix}
        \norm{\boldsymbol{x}^j - \mathbf{d}_1^i}_2^2 \leq r_i\\
        \vdots\\
        \norm{\boldsymbol{x}^j - \mathbf{d}_k^i}_2^2 \leq r_i
        \end{bmatrix}\quad  \forall i \in \{1, 2, \dots, L\},\\
    \end{aligned}
\end{equation}
where $\mathbf{d}_1^i, \dots \mathbf{d}_k^i$ are $n$-dimensional data points belonging to the same group, $\boldsymbol{x}^j$ are the cluster centers, and $\mathcal{K} =\{1,2,\dots k\}$ contains the index of each group, \ie each assigned label. Here we only consider the cases where $L > |\mathcal{K}|$, otherwise, there is a trivial solution to the clustering.

\begin{figure}[t]
    \centering
    \includegraphics[trim={0.1cm 0.1cm 0.1cm 0.1cm},clip,width=.6\linewidth]{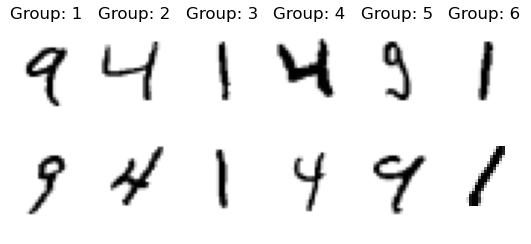}
      \caption{Illustration of a semi-supervised K-means clustering task where the goal is to cluster the 6 groups of images with two images per group. Images are from the MNIST data set \cite{lecun2010mnist}.} 
    \label{fig:relaxations_independent_semi}
\end{figure}

We generate a few problem instances using the  MNIST data set \cite{lecun2010mnist}, and vary the number of data points (images) in each group. We randomly select the data, but such that the data points in each group have the same label and that the number of different labels among the groups matches the desired number of clusters. The characteristics of the instances are shown in Table \ref{tab:test_probs_semi}. The main purpose of these instances is to test the impact of the number of constraints per disjunct on the performance of $P$-split formulations; therefore, the number of groups and clusters is kept constant. We generate 3 different instances of each type by randomly selecting  two labels and data points with these labels from MNIST.

\begin{table}[h]
\centering
\begin{tabular}{ c| c | c |c|c  } \hline
 name &\ data points per group \ & \ \# groups \ & \  \# clusters& \  \# data dimension\\
 \hline
 Semi\_cluster\_1 & 1 & 6 & 2& 784\\ 
Semi\_cluster\_2 & 2 & 6 & 2& 784\\ 
Semi\_cluster\_3 & 3 & 6 & 2& 784\\
Semi\_cluster\_4 & 4 & 6 & 2& 784\\
Semi\_cluster\_5 & 5 & 6 & 2& 784\\
Semi\_cluster\_6 & 6 & 6 & 2& 784\\
Semi\_cluster\_7 & 7 & 6 & 2& 784\\
Semi\_cluster\_8 & 8 & 6 & 2& 784\\\hline

\end{tabular}
\caption{Details of semi-supervised K-means clustering test instances.}
\label{tab:test_probs_semi}
\end{table}

\subsubsection{Computational setup}
We modeled all the formulations using Gurobi's Python interface and solved them using Gurobi. The computational performance of the formulations depends on both the tightness of the continuous relaxation and the computational complexity of the subproblems. To minimize the variations caused by heuristics in Gurobi, we used parameter settings \texttt{MIPFocus} = 3, \texttt{Cuts} = 1, and \texttt{MIQCPMethod} = 1 for all problems. 
We found that using \texttt{PreMIQCPForm} = 2 drastically improves the performance of the extended convex hull formulations for the clustering and P\_ball problems. However, it resulted in worse performance for the other formulations and, therefore, we only used it with the convex hull. Since the NN problems only contain linear constraints, they are only affected by the \texttt{MIPFocus} and \texttt{Cuts} parameters. We applied a time limit of 1800s for all problems.   
Default parameters were used for all other settings and we have used Gurobi version 9. All the clustering and P\_ball problems were solved on a laptop with an i9 9980HK processor and 32GB RAM. The NN problems were solved on desktop computers with i7 8700K processors and 16GB RAM. 

As described in Section~\ref{sec:6}, different variable partitionings can lead to differences in the $P$-split formulations. For the clustering, semi-supervised clustering, and P\_ball problems, we partitioned the variables based on their ordered indices and used the smallest valid M-coefficients and tight globally valid bounds for the $\alpha$-variables. For these problems, the variables are \say{symmetric} in the sense that they have a similar role in the optimization problem and there are no groups of variables that are clearly more important. In such circumstances, the partitioning strategy is likely less important, and initial tests indeed showed little variation from other partitioning strategies. For the NN problems, we obtained bounds using interval arithmetic, or LP-based OBBT, and used the equal size partitioning strategy \cite{tsay2021partition}. For the NN problems, the partitioning strategy is more important as there are variables that play a more prominent role. 

\subsection{Computational results}
Table~\ref{tab:res} summarizes the computational results for the clustering, and P\_ball problems by showing the average time and average number of explored nodes to solve the problems with big-M, $P$-split, and the convex hull formulations. The results show that the $P$-split formulations can have a computational advantage over both the classical big-M formulation and the convex hull formulation. The advantage comes from the continuous relaxation being tighter than the relaxation obtained by big-M, while the subproblems remain computationally cheaper than those originating from the convex hull. The results support the hypothesis that $P$-split formulations can well-approximate the convex hull, as the number of explored nodes are similar with both types of formulations for many of the problems. For several of the problem classes, there is a clear trade-off where the intermediate $P$-split formulations, that are not as tight as possible but with fewer variables and constraints, result in the best computational performance. For the different classes of problems, the optimal value of $P$ varies. However, the various $P$-split formulations perform well overall, and for some classes all the $P$-split formulations we tested outperform both big-M and convex hull formulations. 

The results for the clustering problems are quite interesting. The $P$-split formulations seem particularly well-suited for this class of problems and clearly outperform both the convex hull and big-M formulations in terms of computational time. For these problems, the $n$-split formulation, \ie $|\mathcal{I}_s|=1$, has the same number of variables as the extended convex hull formulation. But, as these problems have nonlinear constraints in the disjuncts, the $n$-split formulation will not necessarily be as strong as the convex hull. However, as the $n$-split formulation effectively linearizes the disjunctions, we end up with simpler constraints. Specifically, to express the convex hull, we need rotated second-order cone constraints, while the $n$-split formulation is represented through linear and convex quadratic constraints.

For the tasks of optimizing the NNs, we did not include the convex hull or $n$-split formulations as these were computationally intractable. Even for the smallest NNs, the $n$-split and convex hull formulations became difficult to handle (it took over 5 minutes to get past the root node). The result that the convex hull formulation for NNs can be computationally inefficient, or even intractable, aligns well with prior observations \cite{anderson2020strong}.

Results for optimizing the NNs are presented in Table~\ref{tab:res_nn}. For the NNs, the smaller $P$-split formulations tend to result in the best computational performance, which can be explained by the size of the problems and their structure. Even for the smallest NNs, we end up with 100 disjunctive constraints and for 50 of these we have 784 variables participating in each disjunctive constraint. Furthermore, the disjunctive constraints appear in a sequential fashion, \ie the outputs of the nodes in the first layer form the input to the second layer. Combined with the fact that tight bounds cannot easily be determined for variables associated with the second layer, we can end up with a weak continuous relaxation even if we use the convex hull of each disjunctive constraint. Therefore, the smaller $P$-split formulations, \ie $P = 2,4$, can have an advantage as the relaxation is tighter than big-M while the subproblems remain of similar size. Using the tighter bounds computed by OBBT reduces the number of explored nodes. The tighter OBBT-based bounds are even more beneficial for the P-split formulations, but overall, we see a similar trend where the smaller $P$-split formulations are the fastest to solve.

Next, we run the test problems presented in Table~\ref{tab:test_probs_semi} with a similar set-up to investigate the performance of $P$-split formulations with varying number of constraints per disjunct. Here we exclude the big-M formulation, as it was clear that it would time out on the majority of the instances. The results are presented in Table~\ref{tab:res2}. The results show similar benefits of the $P$-split formulations, even as the number of constraints per disjoint increases. The convex hull formulations result in fewer explored nodes due to having a stronger continuous relaxation, but as mentioned, this comes at the cost of a computationally more expensive problem formulation. Overall, increasing the number of constraints per disjuncts leads to more difficult problems with more constraints and variables in both the $P$-split and convex hull formulations. Nevertheless, as the number of constraints increases, we still observe a clear advantage of the intermediate $P$-split formulations. 

Overall, the $P$-split formulations perform well 
for the clustering, semi-supervised clustering, and P\_ball problems. For these problems, there is a clear advantage of using an intermediate formulation, and a few variables per partition seems to be the best choice. There is some variation in the choice of $P$ that results in the best performance in terms of computational time, but overall we obtain a good performance with both 2, 3, and 4 variables per partition. As expected, the number of explored nodes decreases with the number of splits as we get stronger relaxations, and for the $n$-split formulations (with only one variable per partition), we are close to the convex hull in terms of the number of nodes explored to solve the problems.

\begin{landscape}
\begin{table}
\centering
\begin{tabular}{ c c | c | c |  c | c | c | c | c } 
\hline
 Problem class & &\ big-M \ & \ $P$-split \ & \ $P$-split & \ $P$-split & \ $P$-split & \ $P$-split & \ convex hull\\
  \multicolumn{2}{r |}{\# vars. per split}&\ \ &  $|\mathcal{I}_s|=8$  & $|\mathcal{I}_s|=4$ & $|\mathcal{I}_s|\approx 3$ & $|\mathcal{I}_s|=2$& $|\mathcal{I}_s|=1$& \\
 \hline
 Cluster\_g1 \ &time & $>$1800 &\cellcolor[gray]{0.8} 13  & \cellcolor[gray]{0.8} 4.5 & \cellcolor[gray]{0.8}4.1 & \cellcolor[gray]{0.8}\textbf{3.6} & \cellcolor[gray]{0.8}6.7 & 61\\
   \multicolumn{2}{r |}{\# nodes\textbackslash gap(\%)} & $100 \%$ & 2105 & 319 & 293 & 127 & \textbf{74} & 112\\
    \hline
     Cluster\_g2 \ &time & $>$1800 &\cellcolor[gray]{0.8} 42  & \cellcolor[gray]{0.8} \textbf{13} & \cellcolor[gray]{0.8} 15 & \cellcolor[gray]{0.8}17 & \cellcolor[gray]{0.8}13 & 402\\
   \multicolumn{2}{r |}{\# nodes\textbackslash gap(\%)} & $100 \%$ & 1515 & 547 & 537 & 522 & \textbf{147} & 160\\
    \hline
 Cluster\_g3 \ &time & $>$1800 & $>$1800  & \cellcolor[gray]{0.8} 840 & \cellcolor[gray]{0.8} 827 & \cellcolor[gray]{0.8}{704} & \cellcolor[gray]{0.8}\textbf{407} & $>$ 1800\\
   \multicolumn{2}{r |}{\# nodes\textbackslash gap(\%)} & $100 \%$ & $4 \%$ & 21821 & 21106 & 21184 & \textbf{18420} & $18 \%$\\
\hline
Cluster\_m1 \ &time & $>$1800 &\cellcolor[gray]{0.8} 50  & \cellcolor[gray]{0.8} 21 & \cellcolor[gray]{0.8} \textbf{20} & \cellcolor[gray]{0.8} 28 & \cellcolor[gray]{0.8} 58 &  1101\\
   \multicolumn{2}{r |}{\# nodes\textbackslash gap(\%)} & $99 \%$ & 2147 & 859 & 617 &443 & 314 & \textbf{310}\\
    \hline
Cluster\_m2 \ &time & $>$1800 &\cellcolor[gray]{0.8} 70  & \cellcolor[gray]{0.8} \textbf{49} & \cellcolor[gray]{0.8} 64  & \cellcolor[gray]{0.8} 101 & \cellcolor[gray]{0.8} 223 &  1339\\
   \multicolumn{2}{r |}{\# nodes\textbackslash gap(\%)} & $100 \%$ & 4101 & 1640 & 1159 & 859& 649 & \textbf{368}\\
    \hline
    Cluster\_m3 \ &time & $>$1800 &\cellcolor[gray]{0.8} \textbf{202}  & \cellcolor[gray]{0.8} \textbf{202} & \cellcolor[gray]{0.8} 254  & \cellcolor[gray]{0.8} 422 & \cellcolor[gray]{0.8} 1028 &  $>$ 1800\\
   \multicolumn{2}{r |}{\# nodes\textbackslash gap(\%)} & $100 \%$ & 8924 & 4035 & 2947 & 2069 & 1512 & $100 \%$\\
   \hline
 P\_ball\_1 \ &time & $>$ 1800 & 1783 & 1140  & 332 &\cellcolor[gray]{0.8} \textbf{81} &\cellcolor[gray]{0.8} 85 & 142 \\
   \multicolumn{2}{r |}{\# nodes\textbackslash gap(\%)} &$100 \%$ & 12882 & 6692 & 6470 & 3247 & 1085 & \textbf{716}\\
    \hline    
P\_ball\_2 \ &time & $>$ 1800 & 750 &\cellcolor[gray]{0.8} 92  &\cellcolor[gray]{0.8} 69 &\cellcolor[gray]{0.8} \textbf{55} &\cellcolor[gray]{0.8} 113 & 121 \\
   \multicolumn{2}{r |}{\# nodes\textbackslash gap(\%)} &$100 \%$ & 21359 & 4853 & 3179 & 1836 & 732 & \textbf{436}\\
\hline
\end{tabular}
\caption{Average time (seconds) and number of explored nodes for solved problems. If more than half of the problems in a class timed out, we report the average remaining gap instead of number of nodes. For the clustering and P\_ball problems, the results are grouped according to the number of variables in each split category. The shaded cells indicate settings where the $P$-split formulations outperformed both the big-M and convex hull formulations in terms of solution time.}
\label{tab:res}
\end{table}
\end{landscape}

\begin{landscape}
\begin{table}
\centering
\begin{tabular}{ c c | c | c |  c | c | c | c | c } 
    \hline
  \rule{0pt}{2.5ex} Problem class &  & \ Big-M & \ $P$-split & \ $P$-split & \ $P$-split & \ $P$-split & \ $P$-split & $P$-split \\
   &  \# splits & \ $P=1$ & \ $P=2$ &\ $P=4$ & \ $P=8$ & \ $P=16$ & $P=32$ & $P=50$ \\
     \hline
NN1\_OSIF \ &time & \textbf{21} & \textbf{21} & 26 & 33 & 77 & 166 & 336 \\
(50 solved) & {\# nodes} & 8039 & 6609 & 6457 & 4520 & 4075 & 3881 & \textbf{2824} \\
     \hline
NN2\_OSIF \ &time & 25 & \cellcolor[gray]{0.8}\textbf{24}  & 32  & 40 & 88 & 175 & 415 \\
(50 solved) & {\# nodes} & 8901 & 6835 & 7102 & 5268 & 4753 & 4508 & \textbf{3145} \\
     \hline
NN3\_OSIF \ &time & 113 &\cellcolor[gray]{0.8} 95  & \cellcolor[gray]{0.8} \textbf{84} & 121 & 305 & 685 & 929 \\
(26 solved) & {\# nodes} & 37650 & 24315 & 12133 & 11853 & 10636 & 10786 & \textbf{10197} \\
    \hline
NN4\_OSIF \ &time & 176 &\cellcolor[gray]{0.8} 131  & \cellcolor[gray]{0.8} \textbf{124} & \cellcolor[gray]{0.8} 172 & 398 & 838 & 1109\\
(25 solved) & {\# nodes} & 54212 & 31064 & 18380 & 14331 & 13139 & 13095 & \textbf{12602} \\
\hline
\hline
NN1\_OSIF* \ &time & 27 & \cellcolor[gray]{0.8}\textbf{10} &\cellcolor[gray]{0.8} 13 &\cellcolor[gray]{0.8} 24 & 58 & 135 & 237 \\
(50 solved) & {\# nodes} & 8846 & 3542 & 3479 & 3991 & 3700 & 3299 & \textbf{2838} \\
     \hline
NN2\_OSIF* \ &time & 28 & \cellcolor[gray]{0.8}\textbf{13} & \cellcolor[gray]{0.8}19 & 33 & 64 & 137 & 275 \\
(50 solved) & {\# nodes} & 9467 & 4843 & 4624 & 4542 & 4213 & 3738 & \textbf{2985}\\
     \hline
NN3\_OSIF* \ &time & 115 & \cellcolor[gray]{0.8}\textbf{47} & \cellcolor[gray]{0.8}64 & \cellcolor[gray]{0.8}105 & 252 & 515 & 826 \\
(26 solved) & {\# nodes} & 33269 & 12454 & 12253 & 10456 & 10602 & \textbf{10048} & 10362 \\
    \hline
NN4\_OSIF* \ &time & 167 & \cellcolor[gray]{0.8}\textbf{66} & \cellcolor[gray]{0.8}98 & \cellcolor[gray]{0.8}156 & 336 & 613 & 1008 \\
(25 solved) & {\# nodes} & 45834 & 16983 & 16999 & 13554 & 12238 & \textbf{10825} & 11514 \\
\hline
\end{tabular}
\caption{Average time (seconds) and number of explored nodes for problems solved by all formulations. The shaded cells indicate settings where the $P$-split formulations outperform both the big-M and convex hull formulations in terms of solution time. In parenthesis, we report the number of instances solved by all formulations. In the first four problem classes, we have used interval arithmetic for computing bounds on the $\alpha$ variables, and the problems marked with $^*$ are the same instances, but with stronger bounds computed by OBBT. The time to compute bounds by OBBT is not included, as the main focus is to compare the formulations with varying bound tightness. }
\label{tab:res_nn}
\end{table}
\end{landscape}

\begin{table}
\centering
\begin{tabular}{ c c  |c |  c | c | c | c | c } 
\hline
 Problem class & \ & \ $P$-split \ & \ $P$-split & \ $P$-split & \ $P$-split & \ $P$-split & \ CH\\
  \multicolumn{2}{r |}{\# vars. per split}&\ $|\mathcal{I}_s|=8$  & $|\mathcal{I}_s|=4$ & $|\mathcal{I}_s|\approx 3$ & $|\mathcal{I}_s|=2$& $|\mathcal{I}_s|=1$& \\
 \hline
Semi\_cluster\_1 &time  &5.5  & 5.5 & \textbf{5.4} & 7.8 &  16 & 131\\
\multicolumn{2}{r |}{\# nodes\textbackslash gap(\%)} & 701 & 220 & 117 & 83 & \textbf{57} & 58\\
    \hline
Semi\_cluster\_2 &time & 61  &  \textbf{28} & \textbf{28} & 38 &  53 & 1418\\
\multicolumn{2}{r |}{\# nodes\textbackslash gap(\%)} & 861  & 340 & 240 & 139 & 92 & \textbf{70}\\
    \hline
Semi\_cluster\_3 &time  & 257  & \textbf{102} &110 & 71 &  209 & 1795\\
\multicolumn{2}{r |}{\# nodes\textbackslash gap(\%)} &781 & 217 & 192 & 159 & 92 & \textbf{43}\\
    \hline
Semi\_cluster\_4 &time  & 565  & 289 & \textbf{173} & 199 &  310 & $> 1800$\\
\multicolumn{2}{r |}{\# nodes\textbackslash gap(\%)} & 763 & 257 & 188 & 146 & \textbf{95} & $70 \%$\\
    \hline
Semi\_cluster\_5 &time  & 843  &  420 & 393 & \textbf{199} &  597 & $> 1800$\\
\multicolumn{2}{r |}{\# nodes\textbackslash gap(\%)} & 625  & 244 & 179 & 167 & \textbf{103} & $76 \%$ \\
    \hline
Semi\_cluster\_6 &time  & 1270  &  685 & 598 & \textbf{371} & 1139 & $> 1800$\\
\multicolumn{2}{r |}{\# nodes\textbackslash gap(\%)} & 717 & 240 & 204 & 163 & \textbf{129} & $85 \%$ \\
    \hline
Semi\_cluster\_7 &time & $> 1800$  &  529 & \textbf{435} & 467 &  1233 & $> 1800$\\
\multicolumn{2}{r |}{\# nodes\textbackslash gap(\%)} & $22 \%$  & 276 & 199 & 177 & \textbf{108} & $92 \%$ \\
    \hline
Semi\_cluster\_8 &time & 1691  &  991 &634 & \textbf{469} &  1493 & $> 1800$\\
\multicolumn{2}{r |}{\# nodes\textbackslash gap(\%)} & 574 & 220 & 208 & 167 & \textbf{108} &  $94 \%$\\
    \hline
\end{tabular}
\caption{Average time (seconds) and number of explored nodes for solved problems with $P$-split formulations and convex hull (CH) formulation. If more than half of the problems in a class time out, we report the average remaining gap instead of the number of nodes. Instead of the number of splits, we list the number of variables in each split partitioning, \textit{i.e.}, $|\mathcal{I}_s|$. As 784 is not divisible by 3, it is not possible to partition the expressions perfectly into groups of 3. Instead we have one partitioning with only one variable and 261 partitionings with 3 variables. The number in the instance name gives the number of constraints per disjoint.}
\label{tab:res2}
\end{table}

\section{Conclusions}\label{sec:8}

We presented $P$-split formulations, a framework for generating a class of mixed-integer formulations of disjunctive constraints. The continuous relaxations of $P$-split formulations are between the relaxations of the big-M and convex hull formulations. We proved that the strength of the $P$-split formulations depends on the number of splits $P$, and that increasing $P$ results in formulations with stronger continuous relaxations. Furthermore, we showed, under some circumstances, the $P$-split formulations converge to an one with equally strong relaxation as the convex hull. The numbers of variables and constraints introduced by $P$-split formulations depend on the numbers of splits and disjunctive terms, rather than on the number of variables present in the disjunctive constraint. The number of splits $P$ is thus a parameter that can be selected to balance the tradeoff between relaxation strength and size of the resulting optimization model.     

The numerical results show a great potential of the $P$-split formulations, by providing a good approximation of the convex hull through a computationally simpler problem. For several the test problems, the intermediate formulations result in a similar number of explored nodes as the convex hull formulation, while reducing the total solution time by an order of magnitude. It is worth mentioning that problems in the computational experiments all have the desired structure described in Assumptions 1--4. Without this structure, the $P$-split formulations do not necessarily have an advantage over either the big-M or convex hull formulation. 
But for problems with a suitable structure, they offer a 
valuable alternative.

\section*{Acknowledgements}
The research was funded by a Newton International Fellowship by the Royal Society (NIF\textbackslash R1\textbackslash 182194) to JK, a grant by the Swedish Cultural Foundation in Finland to JK, a grant by the Swedish Research Council (2022-03502), and by Engineering \& Physical Sciences Research Council (EPSRC) Fellowships to RM and CT (grant numbers EP/P016871/1 and EP/T001577/1). CT also acknowledges support from an Imperial College Research Fellowship. The project was also in-part financially sponsored by  Digital Futures at KTH through JK. Finally, the authors are grateful for the opportunity to spend time at the Oberwolfach workshop “Mixed-integer Nonlinear Optimization: A Hatchery for Modern Mathematics” where we finalized some details of the paper.

\section*{Declarations}
\textbf{Conflict of interest} The authors
have no conflicts of interest to declare that are relevant to the content of this article.

\bibliographystyle{spbasic}       
\bibliography{Ref}  

\begin{thebibliography}{10}

\bibitem{aloise2009np}
D.~Aloise, A.~Deshpande, P.~Hansen, and P.~Popat.
\newblock {NP-hardness of Euclidean sum-of-squares clustering}.
\newblock {\em Machine learning}, 75(2):245--248, 2009.

\bibitem{anderson2020strong}
R.~Anderson, J.~Huchette, W.~Ma, C.~Tjandraatmadja, and J.~P. Vielma.
\newblock Strong mixed-integer programming formulations for trained neural
  networks.
\newblock {\em Mathematical Programming}, pages 1--37, 2020.

\bibitem{bair2013semi}
E.~Bair.
\newblock Semi-supervised clustering methods.
\newblock {\em Wiley Interdisciplinary Reviews: Computational Statistics},
  5(5):349--361, 2013.

\bibitem{balas1974disjunctive}
E.~Balas.
\newblock Disjunctive programming: properties of the convex hull of feasible
  points, gsia management science research report msrr 348, 1974.

\bibitem{balas1979disjunctive}
E.~Balas.
\newblock Disjunctive programming.
\newblock In {\em Annals of Discrete Mathematics}, volume~5, pages 3--51.
  Elsevier, 1979.

\bibitem{balas1985disjunctive}
E.~Balas.
\newblock Disjunctive programming and a hierarchy of relaxations for discrete
  optimization problems.
\newblock {\em SIAM Journal on Algebraic Discrete Methods}, 6(3):466--486,
  1985.

\bibitem{balas1988convex}
E.~Balas.
\newblock On the convex hull of the union of certain polyhedra.
\newblock {\em Operations Research Letters}, 7(6):279--283, 1988.

\bibitem{balas1998disjunctive}
E.~Balas.
\newblock Disjunctive programming: Properties of the convex hull of feasible
  points.
\newblock {\em Discrete Applied Mathematics}, 89(1-3):3--44, 1998.

\bibitem{Balas2018}
E.~Balas.
\newblock {\em Disjunctive Programming}.
\newblock Springer International Publishing, 2018.

\bibitem{balas1993lift}
E.~Balas, S.~Ceria, and G.~Cornu{\'e}jols.
\newblock A lift-and-project cutting plane algorithm for mixed 0--1 programs.
\newblock {\em Mathematical Programming}, 58(1-3):295--324, 1993.

\bibitem{balas2022v}
E.~Balas and A.~M. Kazachkov.
\newblock V-polyhedral disjunctive cuts.
\newblock {\em arXiv preprint arXiv:2207.13619}, 2022.

\bibitem{basu2004probabilistic}
S.~Basu, M.~Bilenko, and R.~J. Mooney.
\newblock A probabilistic framework for semi-supervised clustering.
\newblock In {\em Proceedings of the tenth ACM SIGKDD international conference
  on Knowledge discovery and data mining}, pages 59--68, 2004.

\bibitem{belotti2013bound}
P.~Belotti.
\newblock Bound reduction using pairs of linear inequalities.
\newblock {\em Journal of Global Optimization}, 56(3):787--819, 2013.

\bibitem{belotti2016handling}
P.~Belotti, P.~Bonami, M.~Fischetti, A.~Lodi, M.~Monaci, A.~Nogales-G{\'o}mez,
  and D.~Salvagnin.
\newblock On handling indicator constraints in mixed integer programming.
\newblock {\em Computational Optimization and Applications}, 65(3):545--566,
  2016.

\bibitem{ben2001lectures}
A.~Ben-Tal and A.~Nemirovski.
\newblock {\em Lectures on Modern Convex Optimization: Analysis, Algorithms,
  and Engineering Applications}, volume~2.
\newblock Siam, 2001.

\bibitem{bernal2024convex}
D.~E. Bernal~Neira and I.~E. Grossmann.
\newblock Convex mixed-integer nonlinear programs derived from generalized
  disjunctive programming using cones.
\newblock {\em Computational Optimization and Applications}, 88(1):251--312,
  2024.

\bibitem{bonami2015mathematical}
P.~Bonami, A.~Lodi, A.~Tramontani, and S.~Wiese.
\newblock On mathematical programming with indicator constraints.
\newblock {\em Mathematical programming}, 151(1):191--223, 2015.

\bibitem{botoeva2020efficient}
E.~Botoeva, P.~Kouvaros, J.~Kronqvist, A.~Lomuscio, and R.~Misener.
\newblock Efficient verification of {ReLU}-based neural networks via dependency
  analysis.
\newblock In {\em AAAI-20 Proceedings}, pages 3291--3299, 2020.

\bibitem{brooks2011support}
J.~P. Brooks.
\newblock Support vector machines with the ramp loss and the hard margin loss.
\newblock {\em Operations research}, 59(2):467--479, 2011.

\bibitem{ceccon2022omlt}
F.~Ceccon, J.~Jalving, J.~Haddad, A.~Thebelt, C.~Tsay, C.~D. Laird, and
  R.~Misener.
\newblock {OMLT}: Optimization \& machine learning toolkit.
\newblock {\em Journal of Machine Learning Research}, 23(349):1--8, 2022.

\bibitem{ceria1999convex}
S.~Ceria and J.~Soares.
\newblock Convex programming for disjunctive convex optimization.
\newblock {\em Mathematical Programming}, 86(3):595--614, 1999.

\bibitem{chen2021pyomo}
Q.~Chen, E.~S. Johnson, D.~E. Bernal, R.~Valentin, S.~Kale, J.~Bates, J.~D.
  Siirola, and I.~E. Grossmann.
\newblock {Pyomo.GDP}: an ecosystem for logic based modeling and optimization
  development.
\newblock {\em Optimization and Engineering}, pages 1--36, 2021.

\bibitem{chen2024sparse}
R.~Chen and J.~Luedtke.
\newblock Sparse multi-term disjunctive cuts for the epigraph of a function of
  binary variables.
\newblock {\em Mathematical Programming}, 206(1):357--388, 2024.

\bibitem{conforti2014integer}
M.~Conforti, G.~Cornu{\'e}jols, and G.~Zambelli.
\newblock Integer programming, volume 271 of graduate texts in mathematics,
  2014.

\bibitem{conforti2008compact}
M.~Conforti and L.~A. Wolsey.
\newblock Compact formulations as a union of polyhedra.
\newblock {\em Mathematical Programming}, 114(2):277--289, 2008.

\bibitem{G2sets}
P.~Fr{\"a}nti, R.~Mariescu-Istodor, and C.~Zhong.
\newblock Xnn graph.
\newblock In A.~Robles-Kelly, M.~Loog, B.~Biggio, F.~Escolano, and R.~Wilson,
  editors, {\em Structural, Syntactic, and Statistical Pattern Recognition},
  pages 207--217, Cham, 2016. Springer International Publishing.

\bibitem{franti2018k}
P.~Fr{\"a}nti and S.~Sieranoja.
\newblock K-means properties on six clustering benchmark datasets.
\newblock {\em Applied Intelligence}, 48(12):4743--4759, 2018.

\bibitem{grimstad2019relu}
B.~Grimstad and H.~Andersson.
\newblock {ReLU} networks as surrogate models in mixed-integer linear programs.
\newblock {\em Computers \& Chemical Engineering}, 131:106580, 2019.

\bibitem{grossmann2003generalized}
I.~E. Grossmann and S.~Lee.
\newblock Generalized convex disjunctive programming: Nonlinear convex hull
  relaxation.
\newblock {\em Computational optimization and applications}, 26(1):83--100,
  2003.

\bibitem{gunluk2010perspective}
O.~G{\"u}nl{\"u}k and J.~Linderoth.
\newblock Perspective reformulations of mixed integer nonlinear programs with
  indicator variables.
\newblock {\em Mathematical programming}, 124(1-2):183--205, 2010.

\bibitem{haddad2022verification}
J.~Haddad, M.~Bynum, M.~Eydenberg, L.~Blakely, Z.~Kilwein, F.~Boukouvala, C.~D.
  Laird, and J.~Jalving.
\newblock Verification of neural network surrogates.
\newblock In {\em Computer Aided Chemical Engineering}, volume~51, pages
  583--588. Elsevier, 2022.

\bibitem{helton2009sufficient}
J.~W. Helton and J.~Nie.
\newblock Sufficient and necessary conditions for semidefinite representability
  of convex hulls and sets.
\newblock {\em SIAM Journal on Optimization}, 20(2):759--791, 2009.

\bibitem{hijazi2012mixed}
H.~Hijazi, P.~Bonami, G.~Cornu{\'e}jols, and A.~Ouorou.
\newblock Mixed-integer nonlinear programs featuring “on/off” constraints.
\newblock {\em Computational Optimization and Applications}, 52(2):537--558,
  2012.

\bibitem{horst1996global}
R.~Horst and H.~Tuy.
\newblock {\em Global Optimization approaches}.
\newblock Springer Berlin, Heidelberg, 1996.

\bibitem{huang2005coverage}
C.-F. Huang and Y.-C. Tseng.
\newblock The coverage problem in a wireless sensor network.
\newblock {\em Mobile networks and Applications}, 10(4):519--528, 2005.

\bibitem{huchette2023deep}
J.~Huchette, G.~Mu{\~n}oz, T.~Serra, and C.~Tsay.
\newblock When deep learning meets polyhedral theory: A survey.
\newblock {\em arXiv preprint arXiv:2305.00241}, 2023.

\bibitem{jeroslow1988simplification}
R.~G. Jeroslow.
\newblock A simplification for some disjunctive formulations.
\newblock {\em {European Journal of Operational research}}, 36(1):116--121,
  1988.

\bibitem{jeroslow1984modelling}
R.~G. Jeroslow and J.~K. Lowe.
\newblock {\em Modelling with integer variables}, pages 167--184.
\newblock Springer, 1984.

\bibitem{kronqvist2017reformulations}
J.~Kronqvist, A.~Lundell, and T.~Westerlund.
\newblock Reformulations for utilizing separability when solving convex {MINLP}
  problems.
\newblock {\em Journal of Global Optimization}, pages 571--592, 2018.

\bibitem{kronqvist2020disjunctive}
J.~Kronqvist and R.~Misener.
\newblock A disjunctive cut strengthening technique for convex {MINLP}.
\newblock {\em Optimization and Engineering}, pages 1--31, 2020.

\bibitem{kronqvist2021between}
J.~Kronqvist, R.~Misener, and C.~Tsay.
\newblock Between steps: Intermediate relaxations between big-m and convex hull
  formulations.
\newblock In {\em International Conference on Integration of Constraint
  Programming, Artificial Intelligence, and Operations Research}, pages
  299--314. Springer, 2021.

\bibitem{lasserre2001explicit}
J.~B. Lasserre.
\newblock An explicit exact {SDP} relaxation for nonlinear 0-1 programs.
\newblock In {\em International Conference on Integer Programming and
  Combinatorial Optimization}, pages 293--303. Springer, 2001.

\bibitem{lecun2010mnist}
Y.~LeCun, C.~Cortes, and C.~Burges.
\newblock Mnist handwritten digit database.
\newblock {\em ATT Labs [Online]. Available: http://yann.lecun.com/exdb/mnist},
  2, 2010.

\bibitem{lee2000new}
S.~Lee and I.~E. Grossmann.
\newblock New algorithms for nonlinear generalized disjunctive programming.
\newblock {\em Computers \& Chemical Engineering}, 24(9-10):2125--2141, 2000.

\bibitem{liittschwager1978integer}
J.~Liittschwager and C.~Wang.
\newblock Integer programming solution of a classification problem.
\newblock {\em Management Science}, 24(14):1515--1525, 1978.

\bibitem{lovasz1991cones}
L.~Lov{\'a}sz and A.~Schrijver.
\newblock Cones of matrices and set-functions and 0--1 optimization.
\newblock {\em SIAM Journal on Optimization}, 1(2):166--190, 1991.

\bibitem{luedtke2010integer}
J.~Luedtke, S.~Ahmed, and G.~L. Nemhauser.
\newblock An integer programming approach for linear programs with
  probabilistic constraints.
\newblock {\em Mathematical programming}, 122(2):247--272, 2010.

\bibitem{macqueen1967some}
J.~MacQueen et~al.
\newblock Some methods for classification and analysis of multivariate
  observations.
\newblock In {\em Proceedings of the fifth Berkeley symposium on mathematical
  statistics and probability}, volume~1, pages 281--297. Oakland, CA, USA,
  1967.

\bibitem{maimon2005data}
O.~Maimon and L.~Rokach.
\newblock {\em Data mining and knowledge discovery handbook}, volume~2.
\newblock Springer, 2005.

\bibitem{misener2014antigone}
R.~Misener and C.~A. Floudas.
\newblock {ANTIGONE}: {A}lgorithms for continuous/integer global optimization
  of nonlinear equations.
\newblock {\em Journal of Global Optimization}, 59(2-3):503--526, 2014.

\bibitem{misener2014framework}
R.~Misener and C.~A. Floudas.
\newblock A framework for globally optimizing mixed-integer signomial programs.
\newblock {\em Journal of Optimization Theory and Applications},
  161(3):905--932, 2014.

\bibitem{papageorgiou2018pseudo}
D.~J. Papageorgiou and F.~Trespalacios.
\newblock Pseudo basic steps: bound improvement guarantees from {Lagrangian}
  decomposition in convex disjunctive programming.
\newblock {\em EURO Journal on Computational Optimization}, 6(1):55--83, 2018.

\bibitem{papalexopoulos2021constrained}
T.~Papalexopoulos, C.~Tjandraatmadja, R.~Anderson, J.~P. Vielma, and
  D.~Belanger.
\newblock Constrained discrete black-box optimization using mixed-integer
  programming.
\newblock {\em arXiv preprint arXiv:2110.09569}, 2021.

\bibitem{pytorch}
A.~Paszke, S.~Gross, F.~Massa, A.~Lerer, J.~Bradbury, G.~Chanan, T.~Killeen,
  Z.~Lin, N.~Gimelshein, L.~Antiga, et~al.
\newblock {Pytorch: An} imperative style, high-performance deep learning
  library.
\newblock In {\em Advances in Neural Information Processing Systems}, pages
  8026--8037, 2019.

\bibitem{patel2022neur2sp}
R.~M. Patel, J.~Dumouchelle, E.~Khalil, and M.~Bodur.
\newblock Neur2sp: Neural two-stage stochastic programming.
\newblock {\em Advances in Neural Information Processing Systems},
  35:23992--24005, 2022.

\bibitem{raman1994modelling}
R.~Raman and I.~E. Grossmann.
\newblock Modelling and computational techniques for logic based integer
  programming.
\newblock {\em Computers \& Chemical Engineering}, 18(7):563--578, 1994.

\bibitem{rubin1997solving}
P.~A. Rubin.
\newblock Solving mixed integer classification problems by decomposition.
\newblock {\em Annals of Operations Research}, 74:51--64, 1997.

\bibitem{ruiz2012hierarchy}
J.~P. Ruiz and I.~E. Grossmann.
\newblock A hierarchy of relaxations for nonlinear convex generalized
  disjunctive programming.
\newblock {\em European Journal of Operational Research}, 218(1):38--47, 2012.

\bibitem{sauglam2006mixed}
B.~Sa{\u{g}}lam, F.~S. Salman, S.~Say{\i}n, and M.~T{\"u}rkay.
\newblock A mixed-integer programming approach to the clustering problem with
  an application in customer segmentation.
\newblock {\em {European Journal of Operational Research}}, 173(3):866--879,
  2006.

\bibitem{sawaya2007computational}
N.~W. Sawaya and I.~E. Grossmann.
\newblock Computational implementation of non-linear convex hull reformulation.
\newblock {\em Computers \& Chemical Engineering}, 31(7):856--866, 2007.

\bibitem{serra2020lossless}
T.~Serra, A.~Kumar, and S.~Ramalingam.
\newblock Lossless compression of deep neural networks.
\newblock In {\em Integration of Constraint Programming, Artificial
  Intelligence, and Operations Research}, pages 417--430. Springer, 2020.

\bibitem{sherali1990hierarchy}
H.~D. Sherali and W.~P. Adams.
\newblock A hierarchy of relaxations between the continuous and convex hull
  representations for zero-one programming problems.
\newblock {\em SIAM Journal on Discrete Mathematics}, 3(3):411--430, 1990.

\bibitem{sherali2013reformulation}
H.~D. Sherali and W.~P. Adams.
\newblock {\em A reformulation-linearization technique for solving discrete and
  continuous nonconvex problems}, volume~31.
\newblock Springer Science \& Business Media, 2013.

\bibitem{stubbs1999branch}
R.~A. Stubbs and S.~Mehrotra.
\newblock A branch-and-cut method for 0-1 mixed convex programming.
\newblock {\em Mathematical Programming}, 86(3):515--532, 1999.

\bibitem{Trespalacios2014}
F.~Trespalacios and I.~E. Grossmann.
\newblock {Review of mixed-integer nonlinear and generalized disjunctive
  programming methods}.
\newblock {\em Chemie Ingenieur Technik}, 86(7):991--1012, 2014.

\bibitem{trespalacios2015algorithmic}
F.~Trespalacios and I.~E. Grossmann.
\newblock Algorithmic approach for improved mixed-integer reformulations of
  convex generalized disjunctive programs.
\newblock {\em INFORMS Journal on Computing}, 27(1):59--74, 2015.

\bibitem{trespalacios2015improved}
F.~Trespalacios and I.~E. Grossmann.
\newblock Improved {Big-M} reformulation for generalized disjunctive programs.
\newblock {\em Computers \& Chemical Engineering}, 76:98--103, 2015.

\bibitem{trespalacios2016cutting}
F.~Trespalacios and I.~E. Grossmann.
\newblock Cutting plane algorithm for convex generalized disjunctive programs.
\newblock {\em INFORMS Journal on Computing}, 28(2):209--222, 2016.

\bibitem{tsay2021partition}
C.~Tsay, J.~Kronqvist, A.~Thebelt, and R.~Misener.
\newblock Partition-based formulations for mixed-integer optimization of
  trained {ReLU} neural networks.
\newblock In {\em Advances in Neural Information Processing Systems}, 2021.

\bibitem{vielma2015mixed}
J.~P. Vielma.
\newblock Mixed integer linear programming formulation techniques.
\newblock {\em Siam Review}, 57(1):3--57, 2015.

\bibitem{vielma2019small}
J.~P. Vielma.
\newblock Small and strong formulations for unions of convex sets from the
  cayley embedding.
\newblock {\em Mathematical Programming}, 177(1-2):21--53, 2019.

\bibitem{vielma2010mixed}
J.~P. Vielma, S.~Ahmed, and G.~Nemhauser.
\newblock Mixed-integer models for nonseparable piecewise-linear optimization:
  unifying framework and extensions.
\newblock {\em Operations research}, 58(2):303--315, 2010.

\bibitem{vielma2011modeling}
J.~P. Vielma and G.~L. Nemhauser.
\newblock Modeling disjunctive constraints with a logarithmic number of binary
  variables and constraints.
\newblock {\em Mathematical Programming}, 128(1-2):49--72, 2011.

\bibitem{vigerske2018scip}
S.~Vigerske and A.~Gleixner.
\newblock {SCIP}: {G}lobal optimization of mixed-integer nonlinear programs in
  a branch-and-cut framework.
\newblock {\em Optimization Methods and Software}, 33(3):563--593, 2018.

\bibitem{xu2020fast}
K.~Xu, H.~Zhang, S.~Wang, Y.~Wang, S.~Jana, X.~Lin, and C.-J. Hsieh.
\newblock Fast and complete: Enabling complete neural network verification with
  rapid and massively parallel incomplete verifiers.
\newblock {\em arXiv preprint arXiv:2011.13824}, 2020.

\bibitem{zhao2023model}
S.~Zhao, C.~Tsay, and J.~Kronqvist.
\newblock Model-based feature selection for neural networks: A mixed-integer
  programming approach.
\newblock In {\em International Conference on Learning and Intelligent
  Optimization}, pages 223--238. Springer, 2023.

\end{thebibliography}

\clearpage
\section*{Appendix A}
\label{sec:appendix_a}
In Section~\ref{sec:3}, we omitted some details on the illustrative example~\eqref{eq:example1}. Here we provide more details on determining bounds for the split variables, and for clarity, we give the full formulations.
\subsection*{A.1 Bounds on split variables} 

To make the comparison as fair as possible, we have determined the tightest possible bounds for all the split variables without considering interactions between the split variables. Since the 1-split is equivalent to the big-M formulation, these bounds also give us the smallest valid big-M coefficients. 

By considering both the box constraints and the fact that one of the constraints in the two disjunctions must hold, we can derive stronger valid bounds. For a 1-split formulation the bounds are
\begin{equation}
    \begin{aligned}
    &  \ubar{\alpha}^1_1 = 0, && \bar{\alpha}^1_1 = 4\cdot4^2,\\
    &  \ubar{\alpha}^2_1 = -4\cdot 4, && \bar{\alpha}^2_1 = 4\cdot\sqrt{1/4}.\\
    \end{aligned}
\end{equation}
Due to symmetry, the bounds for any balanced 2-split formulation, \ie with the same number of variables in each split, are 
\begin{equation}
    \begin{aligned}
    &  \ubar{\alpha}^1_s = 0, && \bar{\alpha}^1_s = 2\cdot4^2, && \forall s\in \{1,2\},\\
    & \ubar{\alpha}^2_s = -2\cdot 4, && \bar{\alpha}^2_s = 2\cdot\sqrt{1/2} , && \forall s\in \{1,2\}.
    \end{aligned}
\end{equation}
Finally, for a 4-split formulation the bounds are
\begin{equation}
    \begin{aligned}
    &  \ubar{\alpha}^1_s = 0, && \bar{\alpha}^1_s = 4^2, && \forall s\in \{1,2,3,4\},\\
    &  \ubar{\alpha}^2_s = -4, && \bar{\alpha}^2_s = 1 , && \forall s\in \{1,2,3,4\}.
    \end{aligned}
\end{equation}
Note these bounds on the $\alpha_s^2$ variables are not additive (Definition 3).

In practice, we are often not able to obtain the tightest bounds without running, a potentially expensive, optimization-based bounds tightening procedure. Therefore, an interesting question is: how are the $P$ split formulations affected by weaker bounds? By observing that $-1 \leq x_i, \quad \forall i \in\{1, 2, 3, 4\}$, we can derive the following bounds with simple interval arithmetic 
\begin{equation}
\label{eq:non_t_bound}
    \begin{aligned}
    &  \ubar{\alpha}^1_s = 0, && \bar{\alpha}^1_s = (0.5P^2-3.5P +7)\cdot4^2, && \forall s\in \{1,\dots, P\},\\
    &  \ubar{\alpha}^2_s = -(0.5P^2-3.5P +7)\cdot4, && \bar{\alpha}^2_s = (0.5P^2-3.5P +7) , && \forall s\in \{1, \dots, P\},
    \end{aligned}
\end{equation}
where $P \in \{1,2,4\}$ is the number of splits. Using these weaker bounds in the $P$-split formulations results in the relaxations shown in Figure~\ref{fig:relaxations_independent_ap}. The bounds in \eqref{eq:non_t_bound} are additive, and, as shown in the figure, each split results in a strictly tighter relaxation. Compared to Figure~\ref{fig:relaxations}, we observe that the weaker bounds result in overall weaker relaxations, but the hierarchical nature of the P-split formulations remains. 
\begin{figure}[t]
    \centering
    \begin{subfigure}[t]{.32\linewidth}
    \includegraphics[trim={1.5cm 0cm 2cm 0cm},clip,width=.99\linewidth]{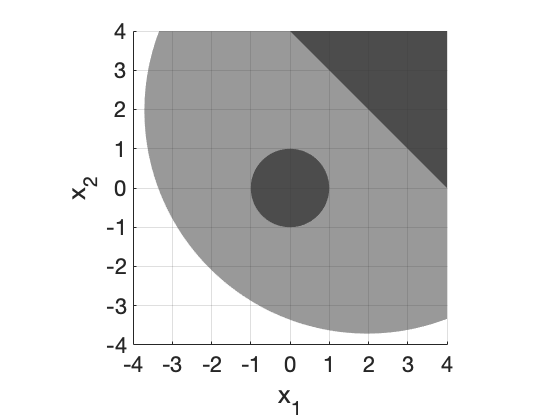}
    \caption{\centering 1-split/big-M $\left(\{x_1, x_2, x_3, x_4\}\right)$}
    \end{subfigure}
    \begin{subfigure}[t]{.32\linewidth}
    \includegraphics[trim={1.5cm 0cm 2cm 0cm},clip,width=.99\linewidth]{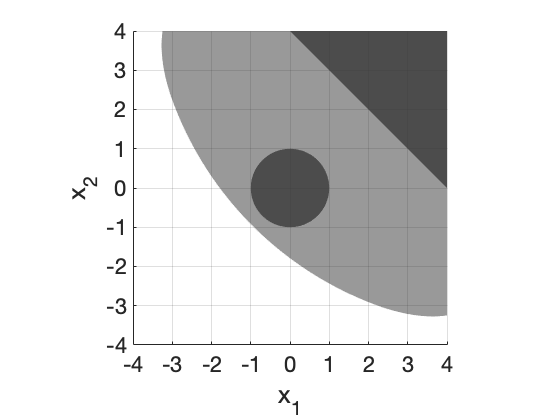}
    \caption{\centering 2-split \newline $\left(\{x_1, x_2\}, \{x_3, x_4\}\right)$}
    \end{subfigure}
    \begin{subfigure}[t]{.32\linewidth}
    \includegraphics[trim={1.5cm 0cm 2cm 0cm},clip,width=.99\linewidth]{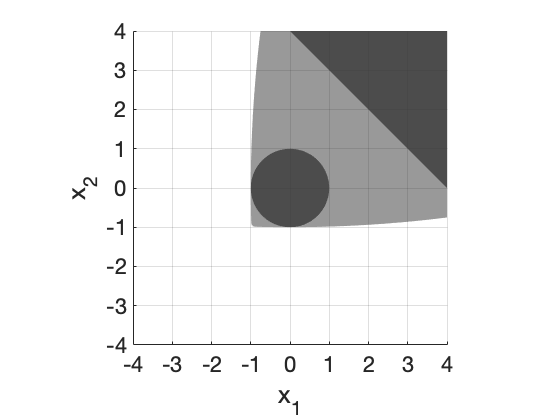}
    \caption{\centering 4-split $\left(\{x_1\},\{x_2\},\{x_3\},\{x_4\}\right)$}
    \end{subfigure}
      \caption{The dark regions show the feasible set of \eqref{eq:example1} in the $x_1,x_2$ space. The light grey areas show the continuously relaxed feasible set of the P-split formulations with the bounds in \eqref{eq:non_t_bound}}. 
    \label{fig:relaxations_independent_ap}
\end{figure}

\subsection*{A.2 Different split formulations}
For the sake of completeness and pedagogical clarity, we provide the full 1-split, 2-split, and 4-split formulations for the illustrative example~\eqref{eq:example1}.

The 1-split formulation is given by
\begin{equation}
    \label{ex1-1-split}
    \begin{aligned}
        & \alpha_1^1 = \nu_1^{\alpha_1^1} + \nu_2^{\alpha_1^1}\\
         & \alpha_1^2 = \nu_1^{\alpha_1^2} + \nu_2^{\alpha_1^2}\\
        &  \nu_1^{\alpha_1^1} \leq \lambda_1\\
        &  \nu_2^{\alpha_1^2} \leq -12\lambda_2\\
        &  \ubar{\alpha}^1_1 \lambda_1\leq \nu_1^{\alpha_1^1} \leq \bar{\alpha}^1_1\lambda_1\\
        &  \ubar{\alpha}^1_1 \lambda_2\leq \nu_2^{\alpha_1^1} \leq \bar{\alpha}^1_1\lambda_2\\
        &  \ubar{\alpha}^2_1 \lambda_1\leq \nu_1^{\alpha_1^2} \leq \bar{\alpha}^2_1\lambda_1\\
        &  \ubar{\alpha}^2_1 \lambda_2\leq \nu_2^{\alpha_1^2} \leq \bar{\alpha}^2_1\lambda_2\\
        & \sum_{i=1}^4 x_i^2 \leq \alpha_1^1, \\ &\sum_{i=1}^4 -x_i \leq \alpha_1^2\\
         & \alpha_1^1, \alpha_1^2, \nu_1^{\alpha_1^1}, \nu_2^{\alpha_1^1},\nu_1^{\alpha_1^2},\nu_2^{\alpha_1^2} \in \mathbb{R}\\
           & \lambda_1, \lambda_2 \in \{0, 1\},  \boldsymbol{x} \in [-4,4]^4. &&          
    \end{aligned}
\end{equation}
Remember as shown by Theorem~\ref{thm:bigM}, the 1-split is equivalent to the big-M in terms of relaxation strength, and can be seen as a \say{complicated} alternative of the big-M formulation.
\vspace{6cm}\\
\phantom{1e}\\

Next the 2-split formulation with the varialbe partitioning $\{x_1,x_2\}$ and $\{x_3,x_4\}$  is given by
\begin{equation}
    \label{ex1-2-split}
    \begin{aligned}
        & \alpha_s^1 = \nu_1^{\alpha_s^1} + \nu_2^{\alpha_s^1} && s \in \{1,2\}\\
         & \alpha_s^2 = \nu_1^{\alpha_s^2} + \nu_2^{\alpha_s^2} && s \in \{1,2\}\\
        &  \sum_{s=1}^2\nu_1^{\alpha_s^1} \leq \lambda_1 &&\\
        &  \sum_{s=1}^2\nu_2^{\alpha_s^2} \leq -12\lambda_2 &&\\
        &  \ubar{\alpha}^1_s \lambda_1\leq \nu_1^{\alpha_s^1} \leq \bar{\alpha}^1_s\lambda_1 && s \in \{1,2\}\\
        &  \ubar{\alpha}^1_s \lambda_2\leq \nu_2^{\alpha_s^1} \leq \bar{\alpha}^1_s\lambda_2 && s \in \{1,2\}\\
        &  \ubar{\alpha}^2_s \lambda_1\leq \nu_1^{\alpha_s^2} \leq \bar{\alpha}^2_s\lambda_1&& s \in \{1,2\}\\
        &  \ubar{\alpha}^2_s \lambda_2\leq \nu_2^{\alpha_s^2} \leq \bar{\alpha}^2_s\lambda_2&& s \in \{1,2\}\\
        & \sum_{i=1}^2 x_i^2 \leq \alpha_1^1, \ \sum_{i=3}^4 x_i^2 \leq \alpha_2^1 \\ &\sum_{i=1}^2 -x_i \leq \alpha_1^2, \ \sum_{i=3}^4 -x_i \leq \alpha_2^2&&\\
        & \alpha_s^1, \alpha_s^2, \nu_1^{\alpha_s^1}, \nu_2^{\alpha_s^1},\nu_1^{\alpha_s^2},\nu_2^{\alpha_s^2} \in \mathbb{R} && s \in \{1,2\} \\ 
        & \lambda_1, \lambda_2 \in \{0, 1\},  \boldsymbol{x} \in [-4,4]^4. &&        
    \end{aligned}
\end{equation}

Finally, the 4-split formulation is given by
\begin{equation}
    \label{ex1-4-split}
    \begin{aligned}
        & \alpha_s^1 = \nu_1^{\alpha_s^1} + \nu_2^{\alpha_s^1} && s \in \{1,2,3,4\}\\
         & \alpha_s^2 = \nu_1^{\alpha_s^2} + \nu_2^{\alpha_s^2} && s \in \{1,2,3,4\}\\
        &  \sum_{s=1}^4\nu_1^{\alpha_s^1} \leq \lambda_1 &&\\
        &  \sum_{s=1}^4\nu_2^{\alpha_s^2} \leq -12\lambda_2 &&\\
        &  \ubar{\alpha}^1_s \lambda_1\leq \nu_1^{\alpha_s^1} \leq \bar{\alpha}^1_s\lambda_1 && s \in \{1,2,3,4\}\\
        &  \ubar{\alpha}^1_s \lambda_2\leq \nu_2^{\alpha_s^1} \leq \bar{\alpha}^1_s\lambda_2 && s \in \{1,2,3,4\}\\
        &  \ubar{\alpha}^2_s \lambda_1\leq \nu_1^{\alpha_s^2} \leq \bar{\alpha}^2_s\lambda_1&& s \in \{1,2,3,4\}\\
        &  \ubar{\alpha}^2_s \lambda_2\leq \nu_2^{\alpha_s^2} \leq \bar{\alpha}^2_s\lambda_2 && s \in \{1,2,3,4\}\\
        &  x_s^2 \leq \alpha_s^1, && s \in \{1,2,3,4\}\\
        & -x_s \leq \alpha_s^2 && s \in \{1,2,3,4\}\\
        & \alpha_s^1, \alpha_s^2, \nu_1^{\alpha_s^1}, \nu_2^{\alpha_s^1},\nu_1^{\alpha_s^2},\nu_2^{\alpha_s^2} \in \mathbb{R} && s \in \{1,2,3,4\} \\ 
        & \lambda_1, \lambda_2 \in \{0, 1\},  \boldsymbol{x} \in [-4,4]^4. &&        
    \end{aligned}
\end{equation}\end{document}